\documentclass{article}

\usepackage[colorlinks,allcolors=blue]{hyperref}
\usepackage{amssymb,amsmath,amsthm,mathrsfs}
\usepackage{venturis2}
\usepackage[T1]{fontenc}
\usepackage{comment}
\usepackage{graphicx, tikz, tikz-3dplot} 
\usetikzlibrary{arrows.meta}

\usepackage{wasysym}
\usepackage{color}
\usepackage{graphicx}	

\def\B{\mathscr B}
\def\C{\mathbb C}
\def\d{\mathrm d}

\def\f{\mathfrak f}
\def\EE{\mathscr E}
\def\F{\mathscr F}
\def\G{\mathcal G}
\def\H{\mathcal H}
\def\Hrond{\mathscr H}
\def\hs{\mathfrak h}
\def\K{\mathscr K}
\def\M{\mathrm M}
\def\N{\mathbb N}
\def\R{\mathbb R}
\def\S{\mathbb S}
\def\SS{\mathscr S}

\def\U{\mathscr U}
\def\V{\mathscr V}
\def\Z{\mathbb Z}

\def\e{\mathop{\mathrm{e}}\nolimits}
\def\im{\mathop{\mathrm{Im}}\nolimits}
\def\re{\mathop{\mathrm{Re}}\nolimits}
\def\Ran{\mathop{\mathrm{Ran}}\nolimits}

\DeclareMathOperator*{\slim}{s\hspace{0.1pt}-\hspace{0.1pt}lim}
\def\ltwo{\mathsf{L}^{\:\!\!2}}
\def\lone{\mathsf{L}^{\:\!\!1}}
\def\Lp{\mathsf{L}^{\:\!\!p}}
\def\linf{\mathsf{L}^{\:\!\!\infty}}

\def\Tr{\mathop{\mathrm{tr}}\nolimits}
\def\dim{\mathop{\mathrm{dim}}\nolimits}
\def\sgn{\mathop{\mathrm{sgn}}\nolimits}
\def\ind{\mathop{\mathrm{ind}}\nolimits}
\def\Index{\mathop{\mathrm{Index}}\nolimits}
\def\det{\mathop{\mathrm{det}}\nolimits}
\def\Wind{\mathop{\mathrm{Wind}}\nolimits}

\newtheorem{Theorem}{Theorem}[section]
\newtheorem{Remark}[Theorem]{Remark}
\newtheorem{Lemma}[Theorem]{Lemma}
\newtheorem{Corollary}[Theorem]{Corollary}
\newtheorem{Proposition}[Theorem]{Proposition}

\begin{document}


\title{Levinson's theorem for two-dimensional scattering systems: \\
it was a surprise, it is now topological~!}

\author{A. Alexander${}^1$\footnote{Supported by an Australian Government RTP scholarship.}, D.~T. Nguyen${}^2$, A. Rennie${}^1$\footnote{Supported by ARC Discovery grant DP220101196.}, S. Richard${}^3$
\footnote{Supported by JSPS Grant-in-Aid for scientific research C no 21K03292.}}

\date{\small}
\maketitle
\vspace{-1cm}

\begin{quote}
\begin{itemize}
\item[1] School of Mathematics and Applied Statistics, University of Wollongong,\\
	Wollongong, Australia
\item[2] Department of Physics, Graduate School of Science, Nagoya University, Furo-cho, Chikusa-ku, Nagoya, 464-8602, Japan
\item[3] Institute for Liberal Arts and Sciences \& Graduate School of Mathematics, Nagoya University, Furo-cho, Chikusa-ku, Nagoya, 464-8601, Japan
\item[] E-mail:  angusa@uow.edu.au, thanhnguyen@st.phys.nagoya-u.ac.jp, renniea@uow.edu.au, \\ richard@math.nagoya-u.ac.jp
\end{itemize}
\end{quote}


\begin{abstract}
We prove a general Levinson's theorem for Schr\"odinger operators in two dimensions with threshold obstructions at zero energy. Our results confirm and simplify earlier seminal results of Boll\'e, Gesztesy et al., while providing an explicit topological interpretation.
We also derive explicit formulas for the wave operators, and so show that they are elements of a $C^*$-algebra introduced by Cordes.
As a consequence of our approach, we provide an evaluation of
the spectral shift function at zero in the presence of $p$-resonances.
\end{abstract}


\smallskip



\section{Introduction}\label{sec_intro}
\setcounter{equation}{0}

Scattering theory for two-dimensional Schr\"odinger operators is a challenging subject
that has been the focus of many studies. 
It is known that the $0$-energy behaviour of the underlying self-adjoint operator
is fairly complicated and plays a crucial role.
Namely, the possible coexistence of a $0$-energy bound state and of two types
of $0$-energy resonances has an impact on propagation properties of the evolution group
and on boundedness of the wave operators in various spaces.
A short review of the corresponding literature is provided below.

The intricate $0$-energy behaviour of the resolvent of two-dimensional 
Schr\"odinger operators has also a tremendous impact on the so-called 
Levinson's theorem. It was shown in \cite{BGDW}, and later confirmed in
\cite{BGD88}, that the presence of an $s$-resonance does not play any role in this context
while the presence of $p$-resonances leads to a contribution similar to bound states.
These properties are in sharp contrast from the one-dimensional or three-dimensional 
situation, and for that reason it was announced as \emph{a surprise}
in the first of the two mentioned papers. 
Unfortunately, the proofs of the results of
\cite{BGD88} are based on double asymptotic expansions of the resolvent, which make
them strenuous to follow.

The aim of the present paper is threefold. We firstly confirm the results
obtained in \cite{BGD88} for Levinson's theorem.
Secondly, we recast their proof in an updated framework
with more powerful tools. Thirdly, we restore the topological nature
of Levinson's theorem by exhibiting it as an index theorem in scattering theory.
This general approach has been described in the review paper \cite{Ric16}
and illustrated in several examples \cite{AR,IR1,IR2,KPR,KR,KR08_1,KR08,KR08_2,KR12,NPR,PR,RT10}. In doing so, we recover the analytic formula for the number of bound states stated in \cite[Thm.~6.3]{BGD88}.

Before presenting our results in more detail, let us provide a brief (and non-exhaustive) description of the literature related to our work.
A decade after the two surprising papers \cite{BGDW,BGD88}, 
a renewed interest in the two-dimensional case has been triggered by
the work \cite{Yaj99} and then \cite{JY02} on the $\Lp$-boundedness of the wave operators. However,
these works were conducted under the assumption that $0$-energy bound states and
$0$-energy resonances are absent (the so-called regular or generic case). The next
breakthrough came with the derivation in \cite{JN01} of a simplified resolvent
expansion, no longer given as a two parameter expansion, but in terms of powers of a
single parameter. Subsequently, numerous works took advantage of this simplified
resolvent expansion, as for example \cite{B16,EG12_0,RT13_2,Sch05} in which the
assumption of the absence $0$-energy bound states and $0$-energy resonances remains. In
other works, it was assumed that $0$-energy bound states and $p$-resonances are absent,
as for example in \cite{T17}, or that only the $p$-resonances are absent, as in
\cite{EGG}. 
The first results on the behaviour of the Schr\"odinger evolution in the general case 
appeared then in \cite{EG12_1}. More recently, two-dimensional
Schr\"odinger operators with point interactions have been investigated: the boundedness of
the wave operators in $\Lp$-spaces in the regular case has been discussed in
\cite{CMY}, while a full picture has been provided in \cite{Yaj20_1}. 
Simultaneously, results on the scattering operator in the general setting
have been exhibited in \cite{RTZ}, together with an analysis of the wave operators
in the absence of $p$-resonance.
In particular, this paper contains the confirmation of the $0$-energy behaviour of the scattering matrix, namely $\lim_{\lambda\searrow 0}S(\lambda)=1$, obtained in \cite{BGD88}
based on the double asymptotic expansion of the resolvent.
Finally, building on  \cite{Yaj20_1}, $\Lp$-boundedness for more general Schr\"odinger operators with threshold obstructions has been fully investigated in \cite{Yaj20_2}.

Our approach for obtaining Levinson's theorem as an index theorem in scattering theory
is based on a detailed study of the wave operators. Let us recall their definition, 
and refer to Section \ref{sec_prelim} for more details.
We consider the scattering system given by the pair of operators $(H,H_0)$, where $H_0$
is the Laplacian in the Hilbert space $\ltwo(\R^2)$ and $H:=H_0+V$ with $V$ a real
potential decaying rapidly at infinity. Under quite general conditions on $V$ it is
known that the wave operators
$$
W_\pm:=\slim_{t\to\pm\infty}\e^{itH}\e^{-itH_0}
$$
exist and are complete. In particular, they are Fredholm operators with no kernel
and with a cokernel given by the subspace spanned by the eigenfunctions
of $H$. This subspace is of finite dimension for sufficiently fast decaying potential.
Another important operator in this context is the scattering operator
defined by $S:=W_+^*W_-$. This operator is unitary in $\ltwo(\R^2)$. 
Since $S$ strongly commutes with $H_0$, the operator $S$
decomposes in the spectral representation of $H_0$.
Thus, if we denote by $\F_0$ the unitary map from $\ltwo(\R^2)$ to $\ltwo\big(\R_+;\ltwo(\S)\big)$ satisfying $\big(\F_0 H_0 f\big)(\lambda) = \lambda \big(\F_0f\big)(\lambda)$ for any $f$
in the domain of $H_0$, then one has
$\F_0 S \F_0^* = \{S(\lambda)\}_{\lambda \in \R_+}$, meaning that $S$ is unitarily
equivalent to a family of unitary operators $\{S(\lambda)\}_{\lambda\in\R_+}$ in
$\ltwo(\S)$. For historical reasons, the operator $S(\lambda)$ is called the
scattering matrix at energy $\lambda$, even though it acts on an infinite-dimensional
Hilbert space $\ltwo(\S)=:\hs$. 

By using the stationary representation of the wave operators, our first result is a new formula for the wave operator $W_-$. More precisely, for $V$ decaying fast enough, we show that the following equality holds:
\begin{equation*}
\F_0\big(W_--1\big)\F_0^* 
=   \big(\tfrac12\big(1-\tanh(\pi A_+)\big)\otimes 1_\hs\big)\big(S(L)-1\big) 
-N_2\;\! \Xi \;\!B  + K,
\end{equation*}
where $A_+$ corresponds to the generator of the dilation group in $\ltwo(\R_+)$, 
$S(L)$ denotes the operator of multiplication by the function $\lambda\mapsto S(\lambda)$,
and $K$ is a compact operator. So far, the product $N_2\;\! \Xi \;\!B$ of three bounded
operators is not really meaningful, but let us stress that this term is non-compact whenever
$H$ admits one or two $p$-resonances at $0$. Note that this formula for $W_-$ is at the root of the topological version of Levinson's theorem, and similar formulas have been obtained in several contexts, see for example \cite{BS,I,IT,IsR,NRT,PR_1,RT13,SB}.

In order to get a better understanding of the new term $N_2\;\! \Xi \;\!B$, 
a new representation is better suited. By conjugation with a suitable unitary rescaling,
the wave operator $W_-$ can be realised on $\ltwo\big(\R;\ltwo(\S)\big)$, and by the
decomposition into even and odd functions on $\R$, we end up studying the
wave operator in the Hilbert space $\ltwo\big(\R_+;\ltwo(\S)\big)^{2}$.  
In this representation, the operator $W_-$ takes the form
\begin{align}\label{eq:big}
\begin{split}
&
\left(\begin{smallmatrix} 
1 & 0 \\ 0 & 1
\end{smallmatrix}\right) 
+\tfrac{1}{2}
\left(\begin{smallmatrix} 
1 &  \tanh\big(\tfrac{\pi}{2}\sqrt{-\Delta_{\rm N}}\big) \phi(A_+) \\ 
\overline{\phi}(A_+)\tanh\big(\tfrac{\pi}{2}\sqrt{-\Delta_{\rm N}}\big)  & 1
\end{smallmatrix}\right) 
\left(\begin{smallmatrix} 
\tilde{S}_{\rm e}(L)-1 & \tilde{S}_{\rm o}(L)  \\ 
\tilde{S}_{\rm o}(L) & \tilde{S}_{\rm e}(L) - 1
\end{smallmatrix}\right) 
\\
& \qquad + \left(\begin{smallmatrix} 
(\tilde{N}_2)_{\rm e}(L) & (\tilde{N}_2)_{\rm o}(L)  \\ 
(\tilde{N}_2)_{\rm o}(L) & (\tilde{N}_2)_{\rm e}(L) 
\end{smallmatrix}\right)  
 \left(\begin{smallmatrix} 
\frac{2}{1+i2A_+} & 0  \\ 
0 & \frac{2}{1+i2A_+}
\end{smallmatrix}\right)  
\left(\begin{smallmatrix} 
\tilde{B}_{\rm e}(L) & \tilde{B}_{\rm o}(L)  \\ 
\tilde{B}_{\rm o}(L) & \tilde{B}_{\rm e}(L) 
\end{smallmatrix}\right)  + K,
\end{split}
\end{align}
with  $\phi(A_+):=-\tanh(\pi A_+)+ i \cosh(\pi A_+)^{-1}$, 
the indices ${\rm e}$ and ${\rm o}$ for the even or odd part of a function defined on $\R$, 
and the tilde functions meaning a rescaling, as for example $\tilde{S}(x):=S\big(\e^{-2x}\big)$
for any $x \in \R$. Again, the operator $K$ is a compact operator.
Let us emphasise some of the specific features of the previous formulas. It involves functions of three natural operators acting on $\ltwo(\R_+)$, namely the Neumann Laplacian $-\Delta_{\rm N}$, the operator $L$ of multiplication by the variable, and the generator $A_+$ of the unitary dilation group. In addition, it is shown in the following sections that all functions involved in this
expression are continuous functions having limits either at $\pm \infty$, or at $0$ and $+\infty$.

Obtaining formula \eqref{eq:big} involves purely analytical tools, starting from the asymptotic expansion of the resolvent provided by \cite{JN01}, and using various analytical tricks for 
studying the stationary formula for the wave operators.
Note that some of these tricks have been suggested by \cite{Yaj20_2}, even if the 
aims and the methods are different.
These investigations are presented in Section \ref{sec_waveop} and in the first half of
Section \ref{sec_rep}. The next key observation is that a $C^*$-algebra $\EE$ generated 
by functions  of the three generators mentioned above has been thoroughly studied
in \cite[Chap.~5]{Cordes}. In particular, a precise description of 
the quotient of this algebra by the set of compact operators is provided: 
the quotient consists of continuous functions defined on the edges of a hexagon
(this hexagon is illustrated in Section \ref{sec_Lev}).
By considering $M_2(\EE)$, the set of $2\times 2$ matrices with values in $\EE$, 
enlarging this algebra by a tensor product with $\K\big(\ltwo(\S)\big)$, and adding a unit,
one ends up with a $C^*$-algebra in which the expression \eqref{eq:big} for the wave operator is natural.

Once in this framework, the rest of the investigation is more algebraic, and is presented
in the second half of Section \ref{sec_rep} and in Section \ref{sec_Lev}.
It firstly consists in computing the image of \eqref{eq:big} in the quotient algebra. 
Since $W_-$ is a Fredholm operator, this image is given by a continuous function $\Gamma$
defined on the edges of the hexagon and taking unitary values in $\C+M_2\big(\K(\ltwo(\S))\big)$. This function is provided in Proposition \ref{prop:6c} and in Lemma \ref{lem:Gamma6}.

The operators $\tilde{N}_2$ and $\tilde{B}$ take
a much more explicit and interesting form in the quotient algebra: together they are the image of a 
projection $P_p$ which is directly linked with the $p$-resonance of $H$. 
Secondly, using a $K$-theoretic argument, the function $\Gamma$ can be linked
to the projection $E_{\rm p}(H)$ on the subspace spanned by the eigenvectors of $H$. 
This construction is presented in Section \ref{sec_Lev} and is based on a description of the
index map borrowed from \cite[Prop.~9.2.4.(ii)]{RLL}. Thirdly, by applying traces, one infers
a numerical equality. This equality involves the index of a Fredholm operator $W_S$ defined by
\begin{equation*}
W_S-1 =   \big(\tfrac12\big(1-\tanh(\pi A_+)\big)\otimes 1_\hs\big)\big(S(L)-1\big),
\end{equation*}
and the operator trace of the bound state projection $E_{\rm p}(H)$. It results in the following equality:
\begin{equation*}
\Index(W_S) + \dim(P_p) = -\# \sigma_{\rm p}(H).
\end{equation*}
Note that the dimension of $P_p$ corresponds to
the number of $p$-resonances, and is $1$ or $2$.
By taking care of the high energy behaviour of the scattering matrix, one finally deduces
the relation
\begin{equation}\label{eq:main_intro}
\frac{1}{2\pi i} \int_0^\infty \Tr\big(S(\lambda)^*S'(\lambda)\big) \, \d \lambda + \frac{1}{4\pi} \int_{\R^2} V(x)\, \d x + \dim(P_p) = -\# \sigma_{\rm p}(H).
\end{equation}
Simultaneously, we also determine the value of the spectral shift function at zero in the presence of $p$-resonances.
We refer to Section \ref{sec_Lev} for more precise statements.

The equality \eqref{eq:main_intro} 
confirms that each $p$-resonance provides a contribution of $1$ to Levinson's theorem, while the $s$-resonance does not provide any contribution. For comparison, let us recall the  version of 
Levinson's theorem obtained in \cite[Thm.~6.3]{BGD88} under the
assumption of exponential decay of the potential and the condition
$\int_{\R^2}V(x)\;\!\d x\ne0$.
In the framework of \cite{BGD88}, Levinson's theorem is expressed as
\begin{equation}\label{eq_eux}
\int_0^\infty\im\big((H-\lambda-i0)^{-1}-(H_0-\lambda-i0)^{-1}\big)\;\!\d\lambda
=-\pi N_-+\pi \;\!\Delta_{-1,-1}-\tfrac14\int_{\R^2}V(x)\;\!\d x,
\end{equation}
where $N_-$ is the number of strictly negative eigenvalues of $H$ and
$-\Delta_{-1,-1}$ is equal to the number of $0$-energy eigenvalues and $p$-resonances.
Taking into account the formal identity \cite[Eq.~(6.45)]{BGD88}:
$$
\im\Tr\big((H-\lambda-i0)^{-1}-(H_0-\lambda-i0)^{-1}\big)
=-\tfrac i2\tfrac\d{\d\lambda}\Tr\big(\ln(S(\lambda))\big),
$$
it follows that \eqref{eq_eux} corresponds to \eqref{eq:main_intro}.

Let us finally mention one main difference between \eqref{eq:main_intro} and \eqref{eq_eux}:
the contribution of the $p$-resonance is not on the same side of the equality, and the same remark
holds for the expression involving the integral of $V$. Our r.h.s.~term contains only (minus) the trace of $E_{\rm p}(H)$, which corresponds to the Fredholm index of $W_-$.
In our approach, the contribution of the $p$-resonance projection $P_p$ is coming from the function $\Gamma$
mentioned above, which describes the image of the wave operator $W_-$ under the quotient map.
In that respect, the $p$-resonance data has to stay on the same side as the scattering operator, which 
is also coming from $\Gamma$. On the other hand, the term involving the integral of $V$
is due to a regularization process for the computation of $\Index(W_S)$.
For that reason, it also has to stay on the l.h.s.~of the equality \eqref{eq:main_intro}.
Altogether, these contribution coming from scattering theory
are equal to the contribution due to index theory, namely (minus) the trace on $E_{\rm p}(H)$.
Even though $\dim(P_p)$ is an integer, moving it to the other side of the equality sign would remove
the topological character of this equality.
As said in the title: it was a surprise, it is now topological~!

\vspace{5mm}

\noindent
{\bf Notations:} $\N:=\{0,1,2,\ldots\}$ is the set of natural numbers, $\SS$ the
Schwartz space on $\R^2$, $\R_+:=(0,\infty)$, and
$\langle\cdot\rangle:=\sqrt{1+|\cdot|^2}$. The sets $\H^s_t$ are the weighted
Sobolev spaces over $\R^2$ with index $s\in\R$ for derivatives and index $t\in\R$ for
decay at infinity \cite[Sec.~4.1]{ABG}, and with shorthand notations $\H^s:=\H^s_0$,
$\H_t:=\H^0_t$, and $\H:=\H^0_0=\ltwo(\R^2)$. For any $s,t\in\R$, the $2$-dimensional
Fourier transform $\F$ is a topological isomorphism of $\H^s_t$ onto $\H^t_s$, and the
scalar product $\langle\cdot,\cdot\rangle_\H$ (antilinear in the first argument)
extends continuously to a duality $\langle\cdot,\cdot\rangle_{\H^s_t,\H^{-s}_{-t}}$
between $\H^s_t$ and $\H^{-s}_{-t}$. Given two Banach spaces $\G_1$ and $\G_2$,
$\B(\G_1,\G_2)$ (resp. $\K(\G_1,\G_2)$) denotes the set of bounded (resp. compact)
operators from $\G_1$ to $\G_2$, with shorthand notation $\B(\G_1):=\B(\G_1,\G_1)$
(resp. $\K(\G_1):=\K(\G_1,\G_1)$). Finally, $\otimes$ stands for the closed tensor
product of Hilbert spaces or the spatial tensor product of operators.

\section{Preliminaries}\label{sec_prelim}
\setcounter{equation}{0}

In this section, we briefly recall some notations and preliminary results
introduced in \cite[Sec.~2]{RTZ}.

\subsection{Free operator}

Set $\hs:=\ltwo(\S)$ and $\Hrond:=\ltwo(\R_+;\hs)$, and let $H_0$ be the (positive)
self-adjoint operator in $\H=\ltwo(\R^2)$ given by minus the Laplacian $-\Delta$ on
$\R^2$. Then, the unitary operator $\F_0:\H\to\Hrond$ defined by
\begin{equation}\label{eq_diag}
\big((\F_0 f)(\lambda)\big)(\omega)=2^{-1/2}(\F f)(\sqrt\lambda\;\!\omega),
\quad f\in\SS,~\lambda\in\R_+,~\omega\in\S,
\end{equation}
is a spectral transformation for $H_0$ in the sense that
$$
\big(\F_0H_0f\big)(\lambda)
=\lambda\;\!\big(\F_0 f\big)(\lambda)
=(L\F_0 f)(\lambda),
\quad\hbox{$f\in\H^2$, a.e. $\lambda\in\R_+$,}
$$
with $L$ the maximal multiplication operator by the variable $\lambda\in\R_+$ in
$\Hrond$. Moreover, for each $\lambda\in\R_+$, the operator $\F_0(\lambda):\SS\to\hs$
given by $\F_0(\lambda)f:=(\F_0f)(\lambda)$ extends to an element of $\B(\H^s_t,\hs)$
for any $s\in\R$ and $t>1/2$, and the function
$\R_+\ni\lambda\mapsto\F_0(\lambda)\in\B(\H^s_t,\hs)$ is continuous.

The asymptotic expansion of $\F_0(\lambda)$ as $\lambda\searrow0$ plays an
important role. By expanding the exponential
$\e^{-i\sqrt\lambda\omega\cdot x}$ in a Taylor series, one gets
\begin{equation}\label{eq_exp_F_0}
\F_0(\lambda)=\gamma_0+\sqrt\lambda\;\!\gamma_1+\lambda\gamma_2+o(\lambda),
\quad\lambda\in\R_+,
\end{equation}
with $\gamma_j:\SS\to\hs$ ($j=0,1,2$) the operator given by
$$
(\gamma_jf)(\omega)
:=\tfrac{(-i)^j}{2^{3/2}\pi\;\!(j!)}\int_{\R^2}\d x\,(\omega\cdot x)^j\;\!f(x),
\quad f\in\SS,~\omega\in\S.
$$
The operator $\gamma_j$ extends to an element of $\B(\H^s_t,\hs)$ for any
$s\in\R$ and $t>j+1$, which implies that the expansion \eqref{eq_exp_F_0} holds in
$\B(\H^s_t,\hs)$ as $\lambda\searrow0$ for any $s\in\R$ and $t>3$. We shall sometimes
use the abbreviated notation $\gamma_2(\lambda)$, or $O(\lambda)$, for the sum
$\lambda\gamma_2+o(\lambda)$ in \eqref{eq_exp_F_0}.

\subsection{Perturbed operator}

Let us now consider a potential $V\in\linf(\R^2;\R)$ satisfying for some $\rho>1$ the
bound
\begin{equation}\label{eq_cond_V}
|V(x)|\le{\rm Const.}\;\!\langle x\rangle^{-\rho},\quad\hbox{a.e. $x\in\R^2$.}
\end{equation}
Then, the perturbed Hamiltonian $H:=H_0+V$ is a short range perturbation of $H_0$, and
it is known that the corresponding wave operators
\begin{equation}\label{eq:wavet}
W_\pm:=\slim_{t\to\pm\infty}\e^{itH}\e^{-itH_0}
\end{equation}
exist and are complete. As a consequence, the scattering operator $S:=W_+^*W_-$ is
unitary in $\H$. Now, define for $z\in\C\setminus\R$ the resolvents of $H_0$ and $H$
$$
R_0(z):=(H_0-z)^{-1}\quad\hbox{and}\quad R(z):=(H-z)^{-1}.
$$
In order to recall properties of $R_0(z)$ and $R(z)$ as $z$ approaches the real axis,
it is convenient to decompose the potential $V$ according to the following rule: for
a.e. $x\in\R^2$ set
\begin{equation}\label{eq:uv}
v(x):=|V(x)|^{1/2}
\quad\hbox{and}\quad
u(x):=
\begin{cases}
+1 & \hbox{if $V(x)\ge0$}\\
-1 & \hbox{if $V(x)<0$,}
\end{cases}
\end{equation}
so that $u$ is self-adjoint and unitary and $V=uv^2$. Then, using the fact that $H$
does not have any positive eigenvalues \cite[Sec.~1]{Kat59} and that a limiting absorption
principle holds for $H_0$ and $H$ \cite[Thm.~4.2]{Agm75}, we infer that the limits
$$
vR_0(\lambda\pm i0)v:=\lim_{\varepsilon\searrow0}vR_0(\lambda\pm i\varepsilon)v
\quad\hbox{and}\quad
vR(\lambda\pm i0)v:=\lim_{\varepsilon\searrow0}vR(\lambda\pm i\varepsilon)v,
$$
exist in $\B(\H)$ and are continuous in the variable $\lambda\in\R_+$. This, together
with the relation
$$
u-uvR(\lambda\pm i\varepsilon)vu=\big(u+vR_0(\lambda\pm i\varepsilon)v\big)^{-1},
\quad\lambda\in\R_+,~\varepsilon>0,
$$
implies the existence and the continuity of the function
$\R_+\ni\lambda\mapsto(u+vR_0(\lambda\pm i0)v)^{-1}\in\B(\H)$. Furthermore, one has
$\lim_{\lambda\to\infty}(u+vR_0(\lambda\pm i0)v)^{-1}=u$ in $\B(\H)$, since
$\lim_{\lambda\to\infty}vR_0(\lambda+i0)v=0$ in $\B(\H)$ \cite[Prop.~7.1.2]{Yaf10}. On
the other hand, the existence in $\B(\H)$ of the limits
$\lim_{\lambda\searrow0}(u+vR_0(\lambda\pm i0)v)^{-1}$ depends on the presence or
absence of eigenvalues or resonances at $0$-energy. This problem has been studied in
detail in \cite{JN01} in dimensions $1$ and $2$. We recall here the main result in
dimension $2$ \cite[Thm.~6.2(ii)]{JN01}: Take $\kappa\in\C^*$ with $\re(\kappa)\ge0$,
let $\eta:=1/\ln(\kappa)$ (with $\ln$ the principal value of the complex logarithm),
and set
$$
\M(\kappa):=u+vR_0(-\kappa^2)v.
$$
Then, if $V$ satisfies \eqref{eq_cond_V} with $\rho>11$ and if $0<|\kappa|<\kappa_0$
with $\kappa_0>0$ small enough, the operator $\M(\kappa)^{-1}$ admits an expansion
\begin{equation}\label{eq_JN}
\M(\kappa)^{-1}
=I_1(\kappa)-g(\kappa)I_2(\kappa)-\tfrac{g(\kappa)\eta}{\kappa^2}\;\!I_3(\kappa),
\end{equation}
with
\begin{align*}
I_1(\kappa)&:=(\M(\kappa)+S_1)^{-1},\\
I_2(\kappa)&:=(\M(\kappa)+S_1)^{-1}S_1(M_1(\kappa)+S_2)^{-1}S_1
(\M(\kappa)+S_1)^{-1},\\
I_3(\kappa)&:=(\M(\kappa)+S_1)^{-1}S_1(M_1(\kappa)+S_2)^{-1}S_2
\big(T_3m(\kappa)^{-1}T_3-T_3 m(\kappa)^{-1}b(\kappa)d(\kappa)^{-1}S_3\nonumber\\
&\quad-S_3d(\kappa)^{-1}c(\kappa)m(\kappa)^{-1}T_3+S_3d(\kappa)^{-1}c(\kappa)
m(\kappa)^{-1}b(\kappa)d(\kappa)^{-1}S_3+S_3d(\kappa)^{-1}S_3\big)\nonumber\\
&\quad\cdot S_2(M_1(\kappa)+S_2)^{-1}S_1(\M(\kappa)+S_1)^{-1},
\end{align*}
and where $S_1\ge S_2\ge S_3$ are orthogonal projections in $\H$, $T_3:=S_2-S_3$,
$g:\C\to\C$ satisfies $g(\kappa)=O(\eta^{-1})$ for $0<|\kappa|<\kappa_0$,
$m:\C\to\B(\H)$ satisfies $m(\kappa)=O(\eta^{-1})$ for $0<|\kappa|<\kappa_0$, and all
other factors are operator-valued functions having limits in $\B(\H)$ as $\kappa\to0$.

One of the initial tasks in \cite{RTZ} has been to provide the expansion near $0$
of an 
operator related to $\M(\kappa)$ which plays an important role for the stationary expression
of the wave operators.
The statement is recalled below, with the convention that $\lambda >0$, 
$\kappa:=-i\sqrt\lambda$ which means that 
$\eta=\tfrac1{\ln(\lambda)/2-i\pi/2}$.

\begin{Theorem}[Thm.~4.7 of \cite{RTZ}]\label{thm_asym_new}
If $V$ satisfies \eqref{eq_cond_V} with $\rho>11$, then one has as $\lambda\searrow0$
\begin{align}\label{eq:asympt}
&\nonumber \big(u+vR_0(\lambda+ i0)v\big)^{-1}v\F_0(\lambda)^* \\
\nonumber & =\tfrac{g(\kappa)\eta}{\sqrt\lambda}\big(T_3-S_3d(\kappa)^{-1}c(\kappa)\big)
m(\kappa)^{-1}T_3v\gamma_1^*+\tfrac1\eta S_3\;\!O(1)+O(1) \\
\nonumber & = \tfrac{\eta}{\sqrt\lambda}S_2\big(T_3-S_3d(\kappa)^{-1}c(\kappa)\big)
g(\kappa)m(\kappa)^{-1}T_3v\gamma_1^*+\tfrac1\eta S_3\;\!O(1)+O(1) \\
& = S_2\Big( \tfrac{\eta}{\sqrt\lambda}\big(T_3-S_3d(\kappa)^{-1}c(\kappa)\big)
g(\kappa)m(\kappa)^{-1}T_3v\gamma_1^*+\tfrac1\eta S_3\;\!O(1)+O(1)
\Big)  + S_2^\bot O(1).
\end{align}
\end{Theorem}

In part of the analysis performed in \cite{RTZ} the assumption $T_3=0$ was imposed.
In this case, the main singularity in this expansion disappears, and only the milder
singular term $\tfrac1\eta S_3\;\!O(1) + O(1)$ remains. Note that the factor $S_3$
is associated with $0$-energy bound states, while the projection $T_3$ is related
to the so called $p$-resonances. In the sequel, we shall remove the assumption that
$T_3=0$.

\section{Stationary expression for the wave operators}\label{sec_waveop}
\setcounter{equation}{0}

In this section, we start by recalling the stationary expression for the wave operator $W_-$.
We then decompose this expression into smaller pieces, which will be analysed separately.
This analysis is taking place in the spectral representation of $H_0$, namely in the space
$\Hrond$.

Since the subsequent developments are based on the asymptotic expansions
provided in \eqref{eq_JN} and in \eqref{eq:asympt}, which hold under the assumption
 \eqref{eq_cond_V} with $\rho>11$, we shall assume this decay in the rest of the paper,
without repeating it. 
Under this assumption, the wave operators \eqref{eq:wavet} obtained by the time-dependent approach and those described by the time-independent approach coincide \cite[Thm.~5.3.6]{Yaf92}.
For suitable $\varphi, \psi \in \Hrond=\ltwo(\R_+;\hs)$, we recall the stationary expression 
for $W_-$, namely
\begin{align}\label{eq:wo}
\nonumber &\big\langle\F_0\big(W_--1\big)\F_0^*\varphi,
\psi\big\rangle_\Hrond \\
&=-\int_\R\d\lambda\,\lim_{\varepsilon\searrow0}\int_0^\infty\d\mu\,
\big\langle\F_0(\mu)v {\mathbf 1}\big(u+vR_0(\lambda+i\varepsilon)v\big)^{-1}v\F_0^*
\delta_\varepsilon(L-\lambda)\varphi,(\mu-\lambda+i\varepsilon)^{-1}
\psi(\mu)\big\rangle_\hs,
\end{align}
where
$$
\delta_\varepsilon(L-\lambda)
:=\tfrac\varepsilon\pi(L-\lambda+i\varepsilon)^{-1}(L-\lambda-i\varepsilon)^{-1}.
$$

Note that we have artificially inserted an identity operator ${\mathbf 1}$ in the above expression.
Indeed, this identity operator can be rewritten as 
$$
{\mathbf 1} =S_2 + S_2^\bot,
$$
which then provides two expressions from \eqref{eq:wo}. The one with $S_2^\bot$ does
not contain any singularity at $0$-energy, and has been thoroughly studied in \cite{RTZ}.

We shall now carefully decouple the low energy and the high energy parts of this expression. 
For that purpose, let us fix $\varrho\in C(\R_+;[0,1])$ with 
\begin{equation*}
\varrho(\lambda)=\begin{cases}0 & 
\lambda<\frac{1}{4}\\ 1 &\lambda>\frac{3}{4}.\end{cases}
\end{equation*} 
The function $1-\varrho:\R_+\to [0,1]$ is denoted by $\varrho^\bot$. 
We shall also use two auxiliary functions $\varrho_1 \in  C(\R_+;[0,1])$ and $\varrho_0 \in  C(\R_+;[0,1])$ with 
\begin{equation*}
\varrho_1(\lambda)=\begin{cases}0 &\lambda<\frac{1}{8}\\
1&\lambda>\frac{1}{4}\end{cases}\quad\mbox{and}\quad 
\varrho_0(\lambda)=\begin{cases}1&\lambda<\frac{3}{4}\\ 0&\lambda>\frac{7}{8}.\end{cases}
 \end{equation*}
As a consequence, $\varrho \varrho_1 = \varrho$ and that $\varrho^\bot \varrho_0 = \varrho^\bot$.
Finally, for $\varphi \in \Hrond$ and $\varepsilon>0$ we also set
$$
\tilde\varphi_\varepsilon(\lambda):=S_2\big(u+vR_0(\lambda+i\varepsilon)v\big)^{-1}v\F_0^*
\delta_\varepsilon(L-\lambda)\varrho^\bot(L)\varphi.
$$ 
Then the term with the factor $S_2$ can be rewritten as
\begin{align*}
& \int_\R\d\lambda\,\lim_{\varepsilon\searrow0}\int_0^\infty\d\mu\,
\big\langle\F_0(\mu)v (\mu-\lambda-i\varepsilon)^{-1} S_2\big(u+vR_0(\lambda+i\varepsilon)v\big)^{-1}v\F_0^*
\delta_\varepsilon(L-\lambda)\varphi,
\psi(\mu)\big\rangle_\hs \\
= 
& \int_0^\infty\d\lambda\,\lim_{\varepsilon\searrow0}\int_0^\infty\d\mu\,
\big\langle\F_0(\mu)v S_2(\mu-\lambda-i\varepsilon)^{-1} S_2\big(u+vR_0(\lambda+i\varepsilon)v\big)^{-1}v\F_0^*
\delta_\varepsilon(L-\lambda)\varrho(L)\varphi,
\psi(\mu)\big\rangle_\hs \\
&+  \int_0^\infty\d\lambda\,\lim_{\varepsilon\searrow0}\int_0^\infty\d\mu\,
\big\langle\F_0(\mu)v S_2\varrho^\bot(\mu)
\frac{1}{\sqrt{\mu}}\big(\sqrt{\mu}-\sqrt{\lambda}+\sqrt{\lambda}\big) (\mu-\lambda-i\varepsilon)^{-1} \tilde\varphi_\varepsilon(\lambda),
\psi(\mu)\big\rangle_\hs \\ 
& +  \int_0^\infty\d\lambda\,\lim_{\varepsilon\searrow0}\int_0^\infty\d\mu\,
\big\langle\F_0(\mu)v S_2\varrho(\mu)
\frac{1}{\mu}\big(\mu-\lambda +\lambda\big)
(\mu-\lambda-i\varepsilon)^{-1} \tilde\varphi_\varepsilon(\lambda),
\psi(\mu)\big\rangle_\hs  \\
=  & \underbrace{\int_0^\infty\d\lambda\,\lim_{\varepsilon\searrow0}\int_0^\infty\d\mu\,
\big\langle\F_0(\mu)v S_2(\mu-\lambda-i\varepsilon)^{-1} S_2\big(u+vR_0(\lambda+i\varepsilon)v\big)^{-1}v\F_0^*
\delta_\varepsilon(L-\lambda)\varrho(L)\varphi,
\psi(\mu)\big\rangle_\hs}_{R_0} \\
& + \underbrace{\int_0^\infty\d\lambda\,\lim_{\varepsilon\searrow0}\int_0^\infty\d\mu\,
\big\langle\F_0(\mu)v S_2\varrho^\bot(\mu)
\frac{1}{\sqrt{\mu}}\big(\sqrt{\mu}-\sqrt{\lambda}\big) (\mu-\lambda-i\varepsilon)^{-1} \tilde\varphi_\varepsilon(\lambda),
\psi(\mu)\big\rangle_\hs}_{R_1} \\ 
& + \underbrace{\int_0^\infty\d\lambda\,\lim_{\varepsilon\searrow0}\int_0^\infty\d\mu\,
\big\langle\F_0(\mu)v S_2\varrho^\bot(\mu)
\frac{1}{\sqrt{\mu}} (\mu-\lambda-i\varepsilon)^{-1} \;\!\sqrt{\lambda}\;\! \tilde\varphi_\varepsilon(\lambda),
\psi(\mu)\big\rangle_\hs}_{R_2} \\ 
& + \underbrace{\int_0^\infty\d\lambda\,\lim_{\varepsilon\searrow0}\int_0^\infty\d\mu\,
\big\langle\F_0(\mu)v S_2\varrho(\mu)
\frac{1}{\mu}\;\!(\mu-\lambda)\;\!
(\mu-\lambda-i\varepsilon)^{-1} \tilde\varphi_\varepsilon(\lambda),
\psi(\mu)\big\rangle_\hs}_{R_3}  \\
& + \underbrace{\int_0^\infty\d\lambda\,\lim_{\varepsilon\searrow0}\int_0^\infty\d\mu\,
\big\langle\F_0(\mu)v S_2\varrho(\mu)
\frac{1}{\mu} (\mu-\lambda-i\varepsilon)^{-1} \;\!\lambda \;\!\tilde\varphi_\varepsilon(\lambda),
\psi(\mu)\big\rangle_\hs}_{R_4}. \\
\end{align*}

These various terms will be treated below. For now, 
let us observe  that for any $\varphi \in C_{\rm c}(\R_+;\hs)$ and any $\lambda>0$, 
the following limit exists in $\H$, as a consequence of \cite[Lem.~2.3]{RT13}~:
$$
\slim_{\varepsilon \searrow 0} \tilde \varphi_\varepsilon(\lambda)
=\varrho^\bot(\lambda)\;\!S_2\big(u+vR_0(\lambda+i0)v\big)^{-1}v\F_0(\lambda)^*\varphi(\lambda).
$$
It is then natural to consider the subsequent operator-valued function of $\lambda$ and study its behaviour for
$\lambda \searrow 0$.
The following statement is mainly a consequence of the expansion \eqref{eq:asympt} and the properties of the function $\varrho^\bot$. 

\begin{Lemma}\label{lem:B3}
The following map is continuous and bounded:
\begin{equation}\label{def:B3}
\R_+\ni \lambda \mapsto
B(\lambda):=\varrho^\bot (\lambda) \sqrt{\lambda}\ln(\lambda)S_2 \Big[\big(u+vR_0(\lambda+ i0)v\big)^{-1}v\F_0(\lambda)^*
\Big] \in \K(\hs,\H).
\end{equation}
The multiplication operator $B:C_{\rm c}(\R_+;\hs)\to\ltwo(\R_+;\H)$ given
by $(B\varphi)(\lambda):=B(\lambda)\varphi(\lambda)$ for
$\varphi\in C_{\rm c}(\R_+;\hs)$ and $\lambda\in\R_+$, extends continuously to an
element of $\B\big(\Hrond,\ltwo(\R_+;\H)\big)$.
\end{Lemma}

\begin{proof}
The continuity of the functions $\lambda\mapsto B(\lambda)$ follows from 
the continuities already mentioned in Section \ref{sec_prelim}.
For the behavior of this function near $0$, we can use Theorem \ref{thm_asym_new}
and replace the term inside the square bracket by the r.h.s.~of \eqref{eq:asympt}. Then, by recalling that $g(\kappa)=O(\eta^{-1})$ and 
$m(\kappa)^{-1}=O(\eta)$ as $\lambda\searrow0$, one infers that the first term behaves as $O\big(\big(\sqrt{\lambda}\ln(\lambda)\big)^{-1}\big)$ in the limit $\lambda\searrow0$.
The factor $\sqrt{\lambda} \ln(\lambda)$ makes this product bounded in the limit $\lambda \searrow 0$. For $\lambda>\frac{3}{4}$, the function $\lambda \mapsto B(\lambda)$ vanishes, since the factor $\varrho^\bot(\lambda)$ has this property. 
The rest of the statement is a direct consequence of boundedness of the map $\lambda \mapsto B(\lambda)$.
\end{proof}

Let us now provide the analysis of the terms $R_0$ to $R_4$ which appear
in our study of the stationary expression for the wave operator $W_-$.
An important role is played by the generator of dilations, which we now describe. We recall
that the dilation group $\{U^+_t\}_{t\in\R}$ in $\ltwo(\R_+)$, with self-adjoint
generator $A_+$, is given by $(U^+_t\varphi)(\lambda):=\e^{t/2}\varphi(\e^t\lambda)$
for $\varphi\in C_{\rm c}(\R_+)$, $\lambda\in\R_+$ and $t\in\R$. Using results of \cite{Jen81}, we can take functions of the operator $A_+$ via the formula
\begin{align}\label{eq:functional-calculus}
[\varphi(A_+) f](x) &= (2\pi)^{-\frac12} \int_{-\infty}^\infty{(\F_1^* \varphi)(t)\;\! [U^+_t f](x) \, \d t},
\end{align}
with $\F_1$ the usual Fourier transform on $\ltwo(\R)$ and $\varphi, f \in C_{\rm c}^\infty(\R)$.
We shall also use the function $\vartheta: \R\to \R$ given for any $s\in \R$ by
\begin{equation}\label{eq:vartheta}
\vartheta(s):=\tfrac12\big(1-\tanh(\pi s)\big).
\end{equation}

\begin{Lemma}\label{lem:R0}
The term $R_0$ can be rewritten as $\big\langle2\pi iN_0\big(\vartheta(A_+)\otimes 1_\H\big)B_0\varphi,\psi\big\rangle_\Hrond$, with $N_0$ and $B_0$
defined below in \eqref{def:N0} and in \eqref{def:B0} respectively.
\end{Lemma}

\begin{proof}
Let us firstly define the map
\begin{equation}\label{def:N0}
\R_+\ni\mu\mapsto N_0(\mu):=\F_0(\mu)vS_2\in\K(\H,\hs).
\end{equation}
It is easy to check that this function is continuous, admits a limit as 
$\mu\searrow0$, and vanishes as $\mu\to\infty$,
see \cite[Lem.~4.9(a)]{RTZ} and its proof for a similar statement.
The multiplication operator $N_0:C_{\rm c}(\R_+;\H)\to\Hrond$ given
by $(N_0\xi)(\mu):=N_0(\mu)\xi(\mu)$ for $\xi\in C_{\rm c}(\R_+;\H)$ and
$\mu\in\R_+$, extends then continuously to an element of
$\B\big(\ltwo(\R_+;\H),\Hrond\big)$. 

Secondly, similar arguments to those of Lemma \ref{lem:B3} show that the map
\begin{equation}\label{def:B0}
\R_+\ni\lambda\mapsto B_0(\lambda)
:= \varrho(\lambda) S_2\big(u+vR_0(\lambda+i0)v\big)^{-1}v\F_0(\lambda)^*\in\K(\hs,\H).
\end{equation}
is continuous and bounded. As a consequence, the multiplication operator $B_0: C_{\rm c}(\R_+;\hs)\to\ltwo(\R_+;\H)$ given
by $(B_0\varphi)(\lambda):=B_0(\lambda)\varphi(\lambda)$ for
$\varphi\in C_{\rm c}(\R_+;\hs)$ and $\lambda\in\R_+$, extends continuously to an
element of $\B\big(\Hrond,\ltwo(\R_+;\H)\big)$.

Then, by considering $\varphi, \psi \in C^\infty_{\rm c}(\R_+)\odot C(\S)$, 
we can prove as in \cite[Thm.~2.5]{RT13_2} that the expression $R_0$ given by
\begin{equation*}
\int_0^\infty\d\lambda\,\lim_{\varepsilon\searrow0}\int_0^\infty\d\mu\,
\big\langle\F_0(\mu)v S_2 (\mu-\lambda-i\varepsilon)^{-1} S_2\big(u+vR_0(\lambda+i\varepsilon)v\big)^{-1}v\F_0^*
\delta_\varepsilon(L-\lambda)\varrho(L)\varphi,
\psi(\mu)\big\rangle_\hs
\end{equation*}
reduces to 
$$
\big\langle2\pi iN_0\big(\vartheta(A_+)\otimes 1_\H\big)B_0\varphi,\psi\big\rangle_\Hrond.
$$
Note that the key argument in the proof is an application of Lebesgue's dominated convergence theorem,
as shown in the proof of \cite[Thm.~2.5]{RT13_2}.
\end{proof}

For the next statement, recall that $\varrho$ is a localization near $\infty$ while
$\varrho_0$ is a localization near $0$, see Section \ref{sec_waveop}.

\begin{Lemma}\label{lem:R3}
The term $R_3$ can be rewritten as $\langle K\varphi,\psi\rangle_\Hrond$, 
with $K\in \K(\Hrond)$.
\end{Lemma}

\begin{proof}
For $\mu, \lambda>0$ and $\varepsilon>0$ we consider the kernel
$$
\Theta_\varepsilon(\mu,\lambda):=\varrho(\mu)\frac{1}{\mu}\;\!(\mu-\lambda)\;\! (\mu-\lambda-i\varepsilon)^{-1}\frac{1}{\sqrt{\lambda}\;\!\ln(\lambda)}\varrho_0(\lambda).
$$
This function defines a Hilbert-Schimdt operator in $\ltwo(\R_+)$ with Hilbert-Schmidt norm satisfying
\begin{equation*}
\|\Theta_\varepsilon\|_{\mathrm{HS}}^2
=\int_0^\infty \d \mu\int_0^\infty \d \lambda |\Theta_\varepsilon(\mu,\lambda)|^2
\leq \int_{\frac{1}{4}}^\infty \d \mu\;\! \Big(\frac{\varrho(\mu)}{\mu}\Big)^2 \int_0^\frac{7}{8} \d\lambda \;\!\Big(\frac{\varrho_0(\lambda)}{\sqrt{\lambda}\;\!\ln(\lambda)}\Big)^2 
\end{equation*}
with an upper bound independent of $\varepsilon$. By an application of the dominated
convergence theorem, one then infers that $\Theta_\varepsilon$
converges in the Hilbert-Schmidt norm to $\Theta_0$ defined 
for $\mu,\lambda>0$ by
$\Theta_0(\mu,\lambda):=\frac{\varrho(\mu)}{\mu}\;\!\frac{\varrho_0(\lambda)}{\sqrt{\lambda}\;\!\ln(\lambda)}$.

By an application of the Lebesgue dominated convergence theorem, one obtains 
as in the previous proof that for $\varphi, \psi \in C^\infty_{\rm c}(\R_+)\odot C(\S)$,
the expression $R_3$ given by 
$$
\int_0^\infty\d\lambda\,\lim_{\varepsilon\searrow0}\int_0^\infty\d\mu\,
\big\langle\F_0(\mu)v S_2\varrho(\mu)
\frac{1}{\mu}\;\!(\mu-\lambda)\;\!
(\mu-\lambda-i\varepsilon)^{-1} \tilde\varphi_\varepsilon(\lambda),
\psi(\mu)\big\rangle_\hs
$$
reduces to 
$$
\big\langle N_0(1_{\ltwo(\R_+)}\otimes S_2) (\Theta_0\otimes 1_\H) B \varphi,\psi \big\rangle_\Hrond,
$$
with $N_0$ defined in \eqref{def:N0} and with $B$ defined in \eqref{def:B3}.
As $S_2$ is finite rank, the product $(1_{\ltwo(\R_+)}\otimes S_2) (\Theta_0\otimes 1_\H)$ corresponds to a compact operator on $\ltwo(\R_+;\H)$. By multiplying this factor by the bounded operators $B$ on the right and $N_0$ on the left, one obtains an element of $\K(\Hrond)$. 
\end{proof}

\begin{Lemma}\label{lem:R4}
The term $R_4$ can be rewritten as $\big\langle2\pi i N_4\big(\vartheta(A_+)\otimes 1_\H\big)
\big(M_4\otimes 1_\H\big)\;\!B\varphi,\psi\big\rangle_\Hrond$,
with $B$ defined in \eqref{def:B3},
and $M_4$ and $N_4$ bounded multiplication operators defined in proof.
\end{Lemma}

\begin{proof}
For $\mu>0$ we set 
\begin{equation}\label{eq:N4}
N_4(\mu):=N_0(\mu)\varrho(\mu)\frac{1}{\mu}.
\end{equation} 
This operator valued function is also bounded, vanishes as $\mu\searrow 0$,  and satisfies $\lim_{\mu\to \infty}N_4(\mu)=0$.
For $\lambda>0$ let us also set 
\begin{equation*}
M_4(\lambda):=\varrho_0(\lambda)\sqrt{\lambda}\frac{1}{\ln(\lambda)}.
\end{equation*} 
This scalar function is clearly bounded, vanishes for $\lambda>\frac{7}{8}$, and satisfies $\lim_{\lambda \searrow 0}N_4(\lambda)=0$. 
The corresponding bounded multiplication operators are denoted by $N_4$ and $M_4$, respectively.

Then, by considering $\varphi, \psi \in C^\infty_{\rm c}(\R_+)\odot C(\S)$, 
we can prove as in \cite[Thm.~2.5]{RT13_2} that the expressions $R_4$ given by
$$
\int_0^\infty\d\lambda\,\lim_{\varepsilon\searrow0}\int_0^\infty\d\mu\,
\big\langle\F_0(\mu)v S_2\varrho(\mu)
\frac{1}{\mu} (\mu-\lambda-i\varepsilon)^{-1} \;\!\lambda \;\!\tilde\varphi_\varepsilon(\lambda),
\psi(\mu)\big\rangle_\hs
$$
reduces to 
$$
\big\langle2\pi iN_4\big(\vartheta(A_+)\otimes 1_\H\big)
\big(M_4\otimes 1_\H\big)\;\!B\varphi,\psi\big\rangle_\Hrond,
$$
with $B$ defined in \eqref{def:B3}.
\end{proof}

\begin{Lemma}\label{lem:R2}
The term $R_2$ can be rewritten as
$\big\langle 2\pi i N_2\;\!\big(\vartheta(A_+)\otimes 1_\H\big)
\big(M_2\otimes 1_\H\big)\;\!B\varphi, \psi\big\rangle_\Hrond$
with $B$ defined in \eqref{def:B3},
and $M_2$ and $N_2$ bounded multiplication operators defined in proof.
\end{Lemma}

\begin{proof}
For $\lambda>0$ let us set $M_2(\lambda):=\varrho_0(\lambda)\frac{1}{\ln(\lambda)}$. This scalar function is clearly bounded, vanishes for $\lambda>\frac{7}{8}$, and satisfies $\lim_{\lambda \searrow 0}M_2(\lambda)=0$. The corresponding bounded multiplication operator in $\ltwo(\R_+)$ are denoted by $M_2$.
Let us also define the map
\begin{equation}\label{eq:N2}
\R_+\ni \mu \mapsto N_2(\mu):= \F_0(\mu) \varrho^\bot(\mu)\frac{1}{\sqrt{\mu}} v S_2 \in \B(\H,\hs).
\end{equation}
Clearly, this map vanishes as $\mu\to \infty$. By the expansion \eqref{eq_exp_F_0} 
together with the algebraic cancellation
$\gamma_0 v S_2=0$, as shown in \cite[Lem.~3.2(b)]{RTZ}, one infers that 
this map  admits a limit as $\mu\searrow 0$. As a consequence, 
the operator valued multiplication operator $N_2:C_{\rm c}(\R_+;\H)\to\Hrond$ given
by $(N_2\xi)(\mu):=N_2(\mu)\xi(\mu)$ for $\xi\in C_{\rm c}(\R_+;\H)$ and
$\mu\in\R_+$, extends then continuously to an element of
$\B\big(\ltwo(\R_+;\H),\Hrond\big)$. 

Then, by considering $\varphi, \psi \in C^\infty_{\rm c}(\R_+)\odot C(\S)$, 
we can prove as in \cite[Thm.~2.5]{RT13_2} that the expressions $R_2$ given by
$$
\int_0^\infty\d\lambda\,\lim_{\varepsilon\searrow0}\int_0^\infty\d\mu\,
\big\langle\F_0(\mu)v S_2\varrho^\bot(\mu)
\frac{1}{\sqrt{\mu}} (\mu-\lambda-i\varepsilon)^{-1} \;\!\sqrt{\lambda}\;\! \tilde\varphi_\varepsilon(\lambda),
\psi(\mu)\big\rangle_\hs
$$
reduces to 
$$
\big\langle 2\pi i N_2\;\!\big(\vartheta(A_+)\otimes 1_\H\big)
\big(M_2\otimes 1_\H\big)\;\!B\varphi, \psi\big\rangle_\Hrond
$$
with $B$ defined in \eqref{def:B3}.
\end{proof}

For the study of the term $R_1$ we need some preparatory results. For that purpose,
let us consider the unitary transformation $\U:\ltwo(\R_+)\to \ltwo(\R)$ defined for 
$f\in \ltwo(\R_+)$ and $x\in \R$ by
\begin{equation}\label{eq:U}
[\U f](x):=\sqrt{2}\e^{-x}f\big(\e^{-2x}\big).
\end{equation}
We also introduce the integral operator $\Xi: \ltwo(\R_+)\to \ltwo(\R_+)$ with kernel
given by 
\begin{equation}\label{eq:Xi}
\Xi(\mu,\lambda):=\varrho_0(\mu)\;\!\frac{1}{\sqrt{\mu}+\sqrt{\lambda}}\;\! \frac{1}{\sqrt{\lambda}\;\!\ln(\lambda)}\;\!\varrho_0(\lambda).
\end{equation}
We can then express $\Xi$ in terms of dilations and position operators.
\begin{Lemma}\label{Lem:Unitary-kernel}
The following equality holds in $\ltwo(\R)$:
$$
\U\;\!\Xi\;\!\U^*
=-\varrho_0\big(\e^{-2X})\;\!\frac{2}{1+i2A}\;\!
\varrho_0\big(\e^{-2X}) + K_1
$$
with $A$ the generator of dilation in $\ltwo(\R)$, $X$ the operator by multiplication
by the variable in $\ltwo(\R)$, with remainder $K_1\in \K\big(\ltwo(\R)\big)$.
\end{Lemma}

\begin{proof}
By a direct computation on any $\f\in \ltwo(\R)$ one has
\begin{align}
\nonumber [\U\;\!\Xi\;\!\U^*\f](x) = & -\varrho_0\big(\e^{-2x})\int_{\R}\frac{\d y}{y}\;\!\frac{1}{1+\e^{x-y}}\varrho_0\big(\e^{-2y}\big)\f(y) \\
\label{eq:s} = & -\varrho_0\big(\e^{-2x})\int_{\R}\frac{\d y}{y}\;\!\Big[\frac{1}{1+\e^{x-y}}-\chi_-(x-y)\Big]\varrho_0\big(\e^{-2y}\big)\f(y) \\
\label{eq:b}  &\ \  -\varrho_0\big(\e^{-2x})\int_{\R}\frac{\d y}{y}\;\!\chi_-(x-y)\;\!\varrho_0\big(\e^{-2y}\big)\f(y),
\end{align}
where $\chi_-$ denotes the characteristic function on $\R_-$.
Observe now that the function 
$$
\R\ni s\mapsto \frac{1}{1+\e^s} -\chi_-(s)\in \R
$$
belongs to $\lone(\R)$. Then, since the function 
$a:\R\ni x\mapsto \varrho_0\big(\e^{-2x}\big)\in \R$ belongs to 
$C_0\big((-\infty,\infty]\big)$ and since 
the function $c:\R\ni y\mapsto \frac{1}{y}\varrho_0\big(\e^{-2y}\big)\in \R$
belongs to $C_0(\R)$,  
one infers that \eqref{eq:s} defines a compact operator $K_1$
of the form $a(X)\;\!b(D)\;\!c(X)$, with $b\in C_0(\R)$ and where $(X,D)$ are the canonically conjugate position and momentum operators on $\ltwo(\R)$.

For \eqref{eq:b}, observe that for $x>0$ one has
\begin{align*}
& -\varrho_0\big(\e^{-2x}\big)\int_{\R}\frac{\d y}{y}\;\!\chi_-(x-y)\;\!\varrho_0\big(\e^{-2y}\big)\f(y) \\
& =  -\varrho_0\big(\e^{-2x}\big)\int_x^\infty\frac{\d y}{y}\;\!\big[\varrho_0\big(\e^{-2X}\big)\f\big](y) \\
& =  -\varrho_0\big(\e^{-2x}\big)\int_0^\infty\d s\;\!\big[\varrho_0\big(\e^{-2X}\big)\f\big]\big(\e^sx\big) \\
& =  -\varrho_0\big(\e^{-2x}\big)\int_{\R}\d  s\;\!\chi_+(s)\e^{-s/2}\Big(U_s\big[\varrho_0\big(\e^{-2X}\big)\f\big]\Big)(x)\,
\end{align*}
where $\chi_+$ denotes the characteristic function of $\R_+$ and $\{U_s\}_{s\in \R}$
corresponds to the dilation group in $\ltwo(\R)$.
Note that we have assumed $x>0$ in the above computation, since for $x<0$ the factor $\varrho_0(\e^{-2x})$ already vanishes. 
Now, since the function $h:\R\ni s\mapsto \chi_+(s)\e^{-s/2} \in \R$ belongs to $\lone(\R)$, 
one infers that  \eqref{eq:b} defines an operator
of the form $a(X)\;\!b(A)\;\!a(X)$,
with the function $a$ introduced above, 
with the function $b$ also belonging to $C_0(\R)$,  
and with $A$ the generator of dilation in $\ltwo(\R)$. 
In fact, the function $b$ can be explicitly computed by taking the Fourier transform of $h$ and using an analogue of \eqref{eq:functional-calculus}, from which we obtain
$b(s):=\frac{2}{1+i2s}$.
\end{proof}

We now consider the term $R_1$.

\begin{Lemma}\label{lem:R1'}
The term $R_1$ can be rewritten as 
$\big\langle \big(N_2\;\! \Xi \;\!B + K\big)\varphi,\psi\big\rangle_\Hrond$
with $N_2$ introduced in \eqref{eq:N2}, $\Xi$ introduced in \eqref{eq:Xi}, 
$B$ introduced in \eqref{def:B3}, and $K\in \K(\Hrond)$.
\end{Lemma}

\begin{proof}
First, we show that for a dense set of $\varphi, \psi\in \Hrond$ and for 
a.e.~$\lambda\in \R_+$ we have
\begin{align*}
& \lim_{\varepsilon\searrow0}\int_0^\infty\d\mu\,
\Big\langle \tilde\varphi_\varepsilon(\lambda),
\big(\sqrt{\mu}-\sqrt{\lambda}\big) (\mu-\lambda+i\varepsilon)^{-1} \;\! 
N_2(\mu)^*\psi(\mu)\Big\rangle_\H \\
& = 
\int_0^\infty\d\mu\,
\Big\langle
\varrho^\bot(\lambda) S_2 \big(u+vR_0(\lambda+ i0)v\big)^{-1}v\F_0(\lambda)^*
\varphi(\lambda), \frac{1}{\sqrt{\mu}+\sqrt{\lambda}}\;\!
N_2(\mu)^*\psi(\mu)\Big\rangle_\H
\end{align*}
where $N_2(\mu)$ has been defined in \eqref{eq:N2}.
This can be obtained by a straightforward application of Lebesgue's dominated convergence theorem
by choosing $\varphi, \psi \in C_{\rm c}(\R_+;\hs)$, and by
observing that 
\begin{align*}
&\Big|\Big\langle \tilde\varphi_\varepsilon(\lambda),
\big(\sqrt{\mu}-\sqrt{\lambda}\big) (\mu-\lambda+i\varepsilon)^{-1} \;\! 
N_2(\mu)^*\psi(\mu)\Big\rangle_\H\Big| \\
& \leq \|\tilde\varphi_\varepsilon(\lambda)\|_\H\;\! \Big|\frac{\sqrt{\mu}-\sqrt{\lambda}}{\mu-\lambda+i\varepsilon}\Big|\;\! \|N_2(\mu)^*\psi(\mu)\|_\H \\
& \leq {\rm Const.}(\lambda)\;\! \frac{1}{\sqrt{\mu}+\sqrt{\lambda}}
\;\!\big\|\|N_2(\cdot)^*\psi(\cdot)\|_\H\big\|_\infty\;\!
\chi_{{\rm supp}\;\! \psi}(\mu), 
\end{align*}
which is clealy $\lone(\R_+)$ as a function of $\mu$. Note that we have used the strong convergence of the map $\varepsilon \mapsto \tilde\varphi_\varepsilon(\lambda)$ in $\H$
to infer the uniform bound on $\|\tilde\varphi_\varepsilon(\lambda)\|_\H$.

The rest of the proof is straightforward. Only for the term $K$
it is necessary to observe that this term is compact  since $\U^* K_1 \U \otimes S_2
\in \K\big(\ltwo(\R_+)\otimes \H\big)$
with $K_1\in \K\big(\ltwo(\R)\big)$ obtained from Lemma \ref{Lem:Unitary-kernel}.
\end{proof}

Let us now consider the final term $R_5$, 
namely the one with the factor $S_2^\bot$ inserted in \eqref{eq:wo}.
For that purpose, we introduce the map
\begin{equation}\label{eq:N5}
\R_+\ni\lambda\mapsto N_5(\lambda):=\F_0(\lambda)v S_2^\bot \in\B(\H,\hs),
\end{equation}
which is continuous, admits a limit as 
$\lambda\searrow0$, and vanishes as $\lambda\to\infty$,
see for example \cite[Lem.~4.8]{RTZ}.
The corresponding multiplication operator in 
$\B\big(\ltwo(\R_+;\H),\Hrond\big)$ is denoted by $N_5$.
Similarly, we define the map
\begin{equation*}
\R_+\ni \lambda \mapsto
B_5(\lambda):=S_2^\bot \big(u+vR_0(\lambda+ i0)v\big)^{-1}v\F_0(\lambda)^*
 \in \B(\hs,\H).
\end{equation*}
This map is continuous, bounded as $\lambda \searrow 0$ thanks to the expansion provided in 
Theorem \ref{thm_asym_new}, and vanishes as $\lambda \to \infty$.
The corresponding multiplication operator in 
 $\B\big(\Hrond,\ltwo(\R_+;\H)\big)$ is denoted by $B_5$.

\begin{Lemma}\label{lem:R5}
The term
$$
R_5:=\int_\R\d\lambda\,\lim_{\varepsilon\searrow0}\int_0^\infty\d\mu\,
\big\langle\F_0(\mu)v S_2^\bot \big(u+vR_0(\lambda+i\varepsilon)v\big)^{-1}v\F_0^*
\delta_\varepsilon(L-\lambda)\varphi,(\mu-\lambda+i\varepsilon)^{-1}
\psi(\mu)\big\rangle_\hs
$$
can be rewritten as
$\big\langle 2\pi i  N_5\;\!\big(\vartheta(A_+)\otimes 1_\H\big)
B_5\varphi, \psi\big\rangle_\Hrond$.
\end{Lemma}

The proof of this Lemma is very similar to the one of Lemma \ref{lem:R0} and 
involves only an application of Lebesgue's dominated convergence theorem,
as shown in the proof of \cite[Thm.~2.5]{RT13_2}.

\section{Various representations}\label{sec_rep}
\setcounter{equation}{0}

In this section we provide two new representations for the expressions obtained
above. These represenations will be useful for considering Levinson's theorem as
an index theorem. 
The first task is perform some commutations with the function $\vartheta(A_+)$ from \eqref{eq:vartheta}
which has appeared several times, The first statement provides the necessary
information for these commutations.

\begin{Lemma}\label{lem:commute}
Let $\vartheta$ be the function introduced in \eqref{eq:vartheta}. Then, for $j\in \{0,2,4,5\}$ the following inclusion holds:
\begin{equation*}
 N_j\big(\vartheta(A_+)\otimes 1_\H\big)-\big(\vartheta(A_+)\otimes 1_\hs\big)N_j
\in \K\big(\ltwo(\R_+;\H),\Hrond\big),
\end{equation*}
with $N_0$ introduced in \eqref{def:N0}, $N_2$ in \eqref{eq:N2}, 
$N_4$ in \eqref{eq:N4},
and $N_5$ in \eqref{eq:N5}.
\end{Lemma}

\begin{proof}
This type of result has already been proved in \cite[Lem.~2.7]{RT13} and is based on
an argument of Cordes, see for instance \cite[Thm.~4.1.10]{ABG}.
The key element is that the functions $\lambda\mapsto N_j(\lambda)$
have limits at $0$ and at $\infty$, and that the
function $s\mapsto\vartheta(s)$ has a limit at
$-\infty$ and at $\infty$. For $\vartheta$ is property is clear, 
and for $N_j$ these properties have been discussed when these functions
have been introduced.
\end{proof}

For the next statement, recall that the scattering operator is defined
by $S:=W_+^*W_-$, and that its representation in the spectral representation of $H_0$
is denoted by $S(L)$, namely $S(L):=\F_0\;\!S\;\!\F_0^*$

\begin{Corollary}
The following equality holds:
\begin{equation*}
R_0+R_2+R_3+R_4+R_5 
= - \big\langle \big(\vartheta(A_+)\otimes 1_\hs\big)\big(S(L)-1\big)\psi\big\rangle_\Hrond +\langle K\varphi,\psi\rangle_\Hrond,
\end{equation*} 
with $K\in \K(\Hrond)$.
\end{Corollary}

\begin{proof}
Let us firstly observe that for any $\lambda>0$ one has
\begin{align*}
& N_0(\lambda)\;\!B_0(\lambda)+N_2(\lambda) \;\!B(\lambda) + N_4(\lambda)\;\! N_4(\lambda) + N_5(\lambda)\;\!B_5(\lambda) \\
& = \F_0(\lambda) v
\Big[ S_2\varrho(\lambda) 
+  S_2\varrho^\bot(\lambda)\frac{1}{\sqrt{\lambda}}\cdot \varrho_0(\lambda)\frac{1}{\ln(\lambda)}
\cdot \varrho^\bot (\lambda) \sqrt{\lambda}\ln(\lambda)
\\
& \qquad \qquad \quad +S_2\varrho(\lambda)\frac{1}{\lambda}\cdot \varrho_0(\lambda)\sqrt{\lambda}\frac{1}{\ln(\lambda)} \cdot \varrho^\bot (\lambda) \sqrt{\lambda}\ln(\lambda) \\
& \qquad \qquad \quad 
+ S_2^\bot 
\Big] \big(u+vR_0(\lambda+ i0)v\big)^{-1}v\F_0(\lambda)^* \\
& =   \F_0(\lambda) v
\Big[S_2 \big(\varrho(\lambda)  + \varrho^\bot(\lambda)^2 
+\varrho^\bot(\lambda)\varrho(\lambda)\big)+S_2^\bot  
\Big] \big(u+vR_0(\lambda+ i0)v\big)^{-1}v\F_0(\lambda)^* \\
& =  \F_0(\lambda) v
\big(u+vR_0(\lambda+ i0)v\big)^{-1}v\F_0(\lambda)^* \\
& =\frac{1}{-2\pi i}\big(S(\lambda)-1\big) 
\end{align*}
where $S(\lambda)$ denotes the scattering matrix at energy $\lambda$.
The last equality can be found for example in \cite[Thm.~1.8.1]{Yaf10}.
It remains then to collect the fomulas obtained in Lemmas \ref{lem:R0}, 
\ref{lem:R3}, \ref{lem:R4}, \ref{lem:R2}, \ref{lem:R5},  together with Lemma
 \ref{lem:commute} to obtain the result.
\end{proof}

Collecting the results obtained so far, we have obtained the equality
\begin{equation}\label{eq:debase}
\F_0\big(W_--1\big)\F_0^* 
=   \big(\tfrac12\big(1-\tanh(\pi A_+)\big)\otimes 1_\hs\big)\big(S(L)-1\big) 
-N_2\;\! \Xi \;\!B  + K,
\end{equation}
with $K\in \K(\Hrond)$. This equality holds in $\Hrond = \ltwo(\R_+;\hs)$.

As suggested by the analysis of the operator $\Xi$ in Lemma \ref{Lem:Unitary-kernel},
we shall firstly look at this equality in the Hilbert space $\ltwo(\R;\hs)$ by using the unitary map
$\U$ defined in \eqref{eq:U}.
The image of $\Xi$ in this representation has been computed in Lemma \ref{Lem:Unitary-kernel}.
For any multiplication operator $M$ defined by $\R_+\ni \lambda\mapsto  M(\lambda)$, it is easily seen that for any $\f\in \ltwo(\R;\hs)$ one has
$$
[\U M \U^*\f](x) = M\big(\e^{-2x}\big)\f(x)\equiv \big[M\big(\e^{-2X}\big)\f\big](x).
$$
Finally, if we consider the dilation group $\{U_t^+\}_{t\in \R}$ we obtain by a straightforward computation that
$$
[\U U_t^+ \U^*\f](x) = \f\big(x-\tfrac{1}{2}t\big) = \big[\e^{-it\frac{1}{2}D}\f\big](x),
$$
where $D=-i\frac{\d}{\d x}$. Summing up this information we find in $\ltwo(\R;\hs)$ the equality
\begin{equation}\label{eq:explicit}
\U\F_0\big(W_--1\big)\F_0^*\U^* 
=   \Big(\tfrac12\big(1-\tanh(\tfrac{\pi}{2}D)\big)\otimes 1_\hs\Big)\big(\tilde{S}(X)-1\big)  
 +\tilde{N_2}(X) \;\!\frac{2}{1+i2A}\;\! \tilde{B}(X)  + K,
\end{equation}
with $K\in \K\big(\ltwo(\R,\hs)\big)$ and with the tilde functions given by rescaling the arguments. More precisely, we set $\tilde{S}(X):=S\big(\e^{-2X}\big)$, 
$\tilde{N_2}(X):=N_2\big(\e^{-2X}\big)$, $\tilde{B}(X):=B\big(\e^{-2X}\big)$.
As clearly visible in this formula, the three generators $X, D, A$ are involved in this expression, 
namely the position operator, the generator of translation, and the generator of dilations.

It turns out that a $C^*$-algebra generated by continuous functions of $3$ generators has been introduced and studied in \cite[Chap.~5]{Cordes}. The algebra is constructed on the Hilbert
space $\ltwo(\R_+)$ while the above expression is taking place on $\ltwo(\R)$. 
In order to fit into the framework of Cordes, we need to consider one more unitary 
transformation, namely the decomposition into even and odd functions on $\ltwo(\R)$.

Let us consider 
$\V:\ltwo(\R)\to\ltwo(\R_+;\C^2)$ given by
$$
\V \f:=\sqrt2
\begin{pmatrix}
\f_{\rm e}\\
\f_{\rm o}
\end{pmatrix}\quad{\rm and}\quad
\big[\V^*
\big(\begin{smallmatrix}
\f_1\\
\f_2
\end{smallmatrix}\big)\big](x)
:=\textstyle\frac1{\sqrt2}
\big[f_1(|x|)+\sgn(x)f_2(|x|)\big],
$$
for $\f\in\ltwo(\R)$, $\big(\begin{smallmatrix}
\f_1\\
\f_2
\end{smallmatrix}\big)\in\ltwo(\R_+;\C^2)$, and $x\in\R$.
Here $\f_{\rm e}, \f_{\rm o}$ denote the even and the odd part of $\f$.
Then, one observes that if $m$ is a function on $\R$ 
\begin{equation}\label{eq:g1}
\V m(X)\V^*=
\left(\begin{smallmatrix}
m_{\rm e}(L)~ & m_{\rm o}(L)\\
m_{\rm o}(L)~ & m_{\rm e}(L)
\end{smallmatrix}\right)
\end{equation}
while 
\begin{equation}\label{eq:g2}
\V m(A)\V^*=
\left(\begin{smallmatrix}
m(A_+) & 0\\
0 & m(A_+)
\end{smallmatrix}\right).
\end{equation}

In order to consider $\V m(D)\V^*$, let us denote by $\F_1$ the usual unitary Fourier transform
in $\ltwo(\R)$, and let $\F_{\rm N}$, $\F_{\rm D}$ be the unitary cosine and sine transforms on $\ltwo(\R_+)$, respectively.  The subscripts $\rm N$ and $\rm D$ are related to the Neumann Laplacian and the Dirichlet Laplacian in $\ltwo(\R_+)$, which are diagonalised by $\F_{\rm N}$ and $\F_{\rm D}$, respectively.  
Note also that these operators correspond to their own inverse.
It is then easily checked that
$$
\V \F_1\V^*=
\left(\begin{smallmatrix}
\F_{\rm N} & 0\\
0 & i\F_{\rm D}
\end{smallmatrix}\right).
$$
In addition, by a straightforward computation one gets
$$
\V m(D)\V^* = \V \F_1^* m(X) \F_1 \V^*
= 
\left(\begin{smallmatrix}
\F_{\rm N} m_{\rm e}(L)\F_{\rm N} &  -i \F_{\rm N} m_{\rm o}(L)\F_{\rm D}\\
i \F_{\rm D}m_{\rm o}(L)\F_{\rm N} & \F_{\rm D} m_{\rm e}(L) \F_{\rm D}
\end{smallmatrix}\right).
$$
For the final step, let us recall that the Neumann Laplacian satisfies $-\Delta_{\rm N}:=\F_{\rm N} L^2 \F_{\rm N}$, and that 
$$
i\F_{\rm N} \F_{\rm D} = -\tanh(\pi A_+)+ i \cosh(\pi A_+)^{-1}=:\phi(A_+).
$$
We refer for example to \cite[Prop.~4.13]{DR} for a proof of the above equality.
Then, we end up with 
\begin{equation}\label{eq:g3}
\V m(D)\V^* = \V \F_1^* m(X) \F_1 \V^*
= 
\left(\begin{smallmatrix}
m_{\rm e}\big(\sqrt{-\Delta_{\rm N}}\big) &  - m_{\rm o}\big(\sqrt{-\Delta_{\rm N}}\big)
\phi(A_+)\\
-\overline{\phi}(A_+) m_{\rm o}\big(\sqrt{-\Delta_{\rm N}}\big) 
&\  \overline{\phi}(A_+) m_{\rm e}\big(\sqrt{-\Delta_{\rm N}}\big) \phi(A_+)
\end{smallmatrix}\right).
\end{equation}
Thus, the three equalities \eqref{eq:g1}, \eqref{eq:g2}, and \eqref{eq:g3} allow us
to compute the image of \eqref{eq:explicit} into $\ltwo(\R_+;\hs)^2$~:

\begin{Lemma}
The expression $\V\U\F_0W_-\F_0^*\U^* \V^*$ is given by
\begin{align}\label{eq:M2value}
\begin{split}
&
\left(\begin{smallmatrix} 
1 & 0 \\ 0 & 1
\end{smallmatrix}\right) 
+\tfrac{1}{2}
\left(\begin{smallmatrix} 
1 &  \tanh\big(\tfrac{\pi}{2}\sqrt{-\Delta_{\rm N}}\big) \phi(A_+) \\ 
\overline{\phi}(A_+)\tanh\big(\tfrac{\pi}{2}\sqrt{-\Delta_{\rm N}}\big)  & 1
\end{smallmatrix}\right) 
\left(\begin{smallmatrix} 
\tilde{S}_{\rm e}(L)-1 & \tilde{S}_{\rm o}(L)  \\ 
\tilde{S}_{\rm o}(L) & \tilde{S}_{\rm e}(L) - 1
\end{smallmatrix}\right) 
\\
& \qquad + \left(\begin{smallmatrix} 
(\tilde{N}_2)_{\rm e}(L) & (\tilde{N}_2)_{\rm o}(L)  \\ 
(\tilde{N}_2)_{\rm o}(L) & (\tilde{N}_2)_{\rm e}(L) 
\end{smallmatrix}\right)  
 \left(\begin{smallmatrix} 
\frac{2}{1+i2A_+} & 0  \\ 
0 & \frac{2}{1+i2A_+}
\end{smallmatrix}\right)  
\left(\begin{smallmatrix} 
\tilde{B}_{\rm e}(L) & \tilde{B}_{\rm o}(L)  \\ 
\tilde{B}_{\rm o}(L) & \tilde{B}_{\rm e}(L) 
\end{smallmatrix}\right)  + K
\end{split}
\end{align}
with $K\in \K\big(\ltwo(\R_+;\hs)^2\big)$.
\end{Lemma}

Let us now recall the already mentioned construction of Cordes. 
In \cite[Sec.~V.7]{Cordes}, the following $C^*$-subalgebra
of $\B\big(L^2(\R_+)\big)$ is introduced: 
$$
\EE:=C^*\Big(a_i(A_+)b_i(L)c_i(-\Delta_{\rm N})\mid a_i\in C\big([-\infty,+\infty]\big), \ b_i,c_i\in C\big([0,+ \infty]\big)\Big).
$$ 
It is then shown in \cite[Thm.~V.7.3]{Cordes} that the quotient algebra $\EE/\K\big(L^2(\R_+)\big)$ is isomorphic to $C( \hexagon)$, the set of continuous functions defined on the edges of a hexagon
see Figure \ref{fig1}. 
For an operator of the form  
$a(A_+)\;\!b(L)\;\!c\big(\sqrt{-\Delta_{\rm N}}\big) \in \EE$, 
its image in the quotient algebra takes the form
\begin{align}\label{eq:6c}
\begin{split}
&\Gamma_1(s):=a(s)\;\!b(0)\;\!c(+\infty), \qquad s\in [-\infty, +\infty], \\
&\Gamma_2(\ell):=a(+\infty)\;\!b(\ell)\;\!c(+\infty), \qquad \ell\in [0,+\infty], \\
&\Gamma_3(\xi):=a(+\infty)\;\!b(+\infty)\;\!c(\xi), \qquad \xi \in [+\infty,0], \\
&\Gamma_4(s):=a(s)\;\!b(+\infty)\;\!c(0), \qquad s\in [+\infty, -\infty], \\
&\Gamma_5(\xi):=a(-\infty)\;\!b(+\infty)\;\!c(\xi), \qquad \xi \in [0,+\infty], \\
&\Gamma_6(\ell):=a(-\infty)\;\!b(\ell)\;\!c(+\infty), \qquad \ell \in [+\infty,0].
\end{split}
\end{align}
Observe that we gave an orientation on the interval on which these functions are
defined. As a result, the concatenation map
$$
\Gamma\equiv (\Gamma_1,\Gamma_2,\Gamma_3,\Gamma_4,\Gamma_5,\Gamma_6): \hexagon \to \C
$$
is continuous, even at the vertices of the hexagon.

Our interest in the construction of Cordes comes from the similarity between the elements of $\EE$
and the formula \eqref{eq:M2value}. Indeed, the functions of the three operators $L$, $A_+$, $-\Delta$ are continuous, and have limits either at $-\infty$ and $+\infty$, or at $0$ and $+\infty$
(we shall recall these limits below). The only difference is that we have to consider
the unital $C^*$-algebra $\big(M_2(\EE)\otimes \K(\hs)\big)^+$, the $2\times 2$ matrices with values in $\EE$ tensor
product with the compact operators on $\hs$, and $\C$ times the identity added. Clearly, this algebra contains the ideal
$M_2\big(\K(\ltwo(\R_+))\big)\otimes \K(\hs)$, and one has
$$
\big(M_2(\EE)\otimes \K(\hs)\big)^+ \big/ M_2\big(\K(\ltwo(\R_+))\big)\otimes \K(\hs)
= \big(M_2\big(C(\hexagon)\big)\otimes \K(\hs)\big)^+.
$$ 
One can thus look at the image of \eqref{eq:M2value} through the quotient map
$$
q: \big(M_2(\EE)\otimes \K(\hs)\big)^+ \to \big(M_2\big(C(\hexagon)\big)\otimes \K(\hs)\big)^+
$$
with kernel $M_2\big(\K(\ltwo(\R_+))\big)\otimes \K(\hs)$.
In the next statement we provide this image, keeping the convention provided in \eqref{eq:6c}
for the enumeration of the $6$ components.

\begin{figure}[!ht]
\centering
\tdplotsetmaincoords{80}{0}
    \begin{tikzpicture}[scale = 4, tdplot_main_coords]
        \coordinate (O) at (0,0,0);
        
        \coordinate (A1) at (1,0,0);
        \coordinate (A2) at (0.4,1.5,0);
        \coordinate (A3) at (1.4,1.5,0);
        \coordinate (B0) at (0,0,1);
        \coordinate (B1) at (1,0,1);
        \coordinate (B2) at (0.4,1.5,1);
        \coordinate (B3) at (1.4,1.5,1);

        \draw[-{Triangle[length=2mm, width=2mm]}, line width = 0.75mm] (A1) node[inner sep = 1.5mm, anchor=north east]{$\infty$} to node [below]{$\ell$} (O) node[inner sep = 1mm, anchor=120]{$0$};
        \draw[dashed, gray] (O) -- (A2) ; %
        \draw[-{Triangle[length=2mm, width=2mm]},line width = 0.75mm] (A3) node[inner sep = 1.5mm, anchor=80]{$0$} to node [below]{$\xi$} (A1) node[inner sep = 0.5mm, anchor=north west]{$\infty$};
        \draw[gray, dashed](A2) -- (A3); %
        \draw[-{Triangle[length=2mm, width=2mm]},line width = 0.75mm] (B0) node[inner sep = 1mm, anchor=250]{$0$} to node [above]{$\ell$} (B1) node[anchor=south east]{$\infty$};
        \draw[gray] (B0) -- (B2);  %
        \draw[-{Triangle[length=2mm, width=2mm]}, line width = 0.75mm] (B1) node[inner sep = 2.5mm, anchor=260]{$\infty$} to node [above]{$\xi$} (B3) node[inner sep = 0.5mm, anchor=300]{$0$};
        \draw[gray] (B2) -- (B3);  %
        \draw[-{Triangle[length=2mm, width=2mm]},line width = 0.75mm] (O) node[inner sep = 0.5mm, anchor=-20]{$-\infty$} to node [left]{$s$} (B0) node[inner sep = 1mm, anchor=30]{$\infty$};
        \draw[gray] (A1) -- (B1);  %
        \draw[gray,dashed] (A2) -- (B2);  %
        \draw[-{Triangle[length=2mm, width=2mm]},line width = 0.75mm] (B3) node[inner sep = 1mm, anchor=150]{$\infty$} to node [right]{$s$} (A3) node[inner sep = 0.75mm, anchor=200]{$-\infty$};
    \end{tikzpicture}
\caption{Representation of the quotient algebra, with orientation indicated on the edges. The starting point of $\Gamma_1$ is located on the lower left corner.}
\label{fig1}
\end{figure}
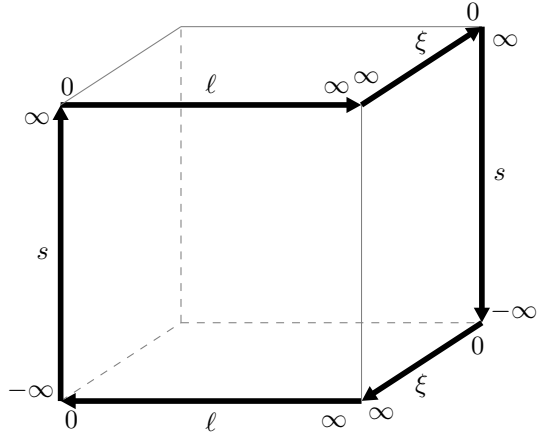

\begin{Proposition}\label{prop:6c}
The operator provided in \eqref{eq:M2value} belongs to $\big(M_2(\EE)\otimes \K(\hs)\big)^+$, 
and its image through the quotient map $q$ consists in the following $6$ operator-valued functions:
\begin{align*}
&\Gamma_1(s):= 
\left(\begin{smallmatrix} 
1 & 0 \\ 0 & 1
\end{smallmatrix}\right) +  \tfrac{1}{2} \big(S(1)-1\big)
 \left(\begin{smallmatrix} 
1 & \phi(s) \\ \overline{\phi}(s) & 1
\end{smallmatrix}\right),&
 s\in [-\infty, +\infty], \\
& \Gamma_2(\ell):= 
\left(\begin{smallmatrix} 
1 & 0 \\ 0 & 1
\end{smallmatrix}\right) + \tfrac{1}{2}\big(S\big(\e^{2\ell}\big)-1\big)
\left(\begin{smallmatrix} 
1 & -1 \\ -1 & 1
\end{smallmatrix}\right), &
 \ell\in [0,+\infty], \\
&\Gamma_3(\xi):= \left(\begin{smallmatrix} 1 & 0 \\ 0 & 1 
\end{smallmatrix}\right),&  \xi \in [+\infty,0], \\
&\Gamma_4(s):= 
\left(\begin{smallmatrix} 
1 & 0 \\ 0 & 1
\end{smallmatrix}\right) 
+\tfrac{1}{2}\frac{2}{1+i2s}N_2(0)\;\!B(0)
\left(\begin{smallmatrix} 1 & 1 \\1 & 1 
\end{smallmatrix}\right),&
 s\in [+\infty, -\infty], \\
&\Gamma_5(\xi):=\left(\begin{smallmatrix} 1 & 0 \\ 0 & 1 
\end{smallmatrix}\right),&   \xi \in [0,+\infty], \\
&\Gamma_6(\ell):= 
\left(\begin{smallmatrix} 
1 & 0 \\ 0 & 1
\end{smallmatrix}\right) + \tfrac{1}{2}\big(S\big(\e^{-2\ell}\big)-1\big)
\left(\begin{smallmatrix} 
1 & 1 \\ 1 & 1
\end{smallmatrix}\right), &
\qquad \ell \in [+\infty,0].
\end{align*} 
\end{Proposition}

\begin{proof}
When the various factors of  \eqref{eq:M2value} were introduced,
their continuity properties and the existence of their limits at endpoints have
been discussed. The only missing information is about
$S(\lambda)-1$. It is known that for any $\lambda>0$, one has
$S(\lambda)-1\in \K(\hs)$, and that the map $\lambda \mapsto S(\lambda)-1$
is continuous, see for example \cite[Prop.~8.1.5]{Yaf10}.
In addition, it has been shown in \cite[Thm.~1.1]{RTZ} that
$\lim_{\lambda \searrow 0}S(\lambda)=1$. Since $\lim_{\lambda \to \infty}S(\lambda)=1$, 
with the limit taken in $\B(\hs)$,
the function $\lambda \mapsto S(\lambda)-1$ belongs to $C_0\big(\R_+,\K(\hs)\big)$.
By inspection of the various factors, one can now infer that
the operator provided in \eqref{eq:M2value} belongs to $\big(M_2(\EE)\otimes \K(\hs)\big)^+$. 
 
Let's move to the image of this operator in the quotient algebra. 
By using the formulas proposed in \eqref{eq:6c}, the computations are rather straightforward. 
For $\Gamma_1$, it is necessary to observe that $\tilde{S}_{\rm e}(0)=S(1)$ while
$\tilde{S}_{\rm o}(0)=0$. In addition, $\lim_{\xi \to + \infty}\tanh\big(\frac{\pi}{2}\xi\big) =1$.
Because of the localization function  $\varrho^\bot$, one also observes that $\tilde{N}_2(0)=0$ and
$\tilde{B}(0)=0$.  
For $\Gamma_2$, note that $\lim_{s\to +\infty}\phi(s)=-1$, and then the expression
$$
\left(\begin{smallmatrix} 
1 & 0 \\ 0 & 1
\end{smallmatrix}\right) 
+\tfrac{1}{2}
\left(\begin{smallmatrix} 
1 &  -1 \\ 
-1  & 1
\end{smallmatrix}\right) 
\left(\begin{smallmatrix} 
\tilde{S}_{\rm e}(\ell)-1 & \ \tilde{S}_{\rm o}(\ell)  \\ 
\tilde{S}_{\rm o}(\ell) & \ \tilde{S}_{\rm e}(\ell) - 1
\end{smallmatrix}\right) 
$$
leads directly to the result. The computation for $\Gamma_6$ is very similar, once the equality
$\lim_{s\to -\infty}\phi(s)=1$ is taken into account.
For $\Gamma_3$ and for $\Gamma_5$, it is sufficient to remember that $\lim_{\lambda \searrow   0}S(\lambda)=1$ and that $\lim_{\lambda \to \infty}S(\lambda)=1$. These equalities
imply that $\lim_{\lambda \to \pm \infty} \tilde{S}_{\rm e}(\ell)=1$ while $\lim_{\lambda \to \pm \infty}\tilde{S}_{\rm o}(\ell)=0$, with these limits taken in $\B(\hs)$.
Finally for $\Gamma_4$, it is necessary to observe that $\lim_{\ell \to\infty}\tilde{N}_2(\ell)
=N_2(0)$, $\lim_{\ell\to \infty}\tilde{B}(\ell)=B(0)$, and then we have
$$
\lim_{\ell \to \infty}(\tilde{N}_2)_{\rm e}(\ell)=\tfrac{1}{2}N_2(0) =\lim_{\ell \to \infty}(\tilde{N}_2)_{\rm o}(\ell), 
$$
and 
$$
\lim_{\ell \to \infty}(\tilde{B})_{\rm e}(\ell)=\tfrac{1}{2}B(0) = \lim_{\ell \to \infty}(\tilde{B})_{\rm o}(\ell).
$$
This leads us directly to the statement. 
\end{proof}

In order to fully exploit the previous result, it remains to compute the terms appearing in $\Gamma_4$.
So, we first determine the expression $B(0):=\lim_{\lambda \searrow 0}B(\lambda)$ 
explicitly by using the results of \cite[Thm.~6.1 and Thm.~6.2]{JN01}. We denote by $X_j$ the multiplication operator by $x_j$ in $\ltwo(\R^2)$,
and recall from \cite[Thm.~6.2(i)]{JN01} that $\Ran(T_3)$ is spanned by 
\begin{equation*}
Q_j = S_2 X_j v
\end{equation*}
for $j = 1,2$. Note that one or both of the $Q_j$ may vanish or they may be linearly dependent, in which case 
the dimension of $\Ran(T_3)$ is strictly smaller than 2.

\begin{Lemma}
The following equality holds:
\begin{align*}
B(0) & =  2 \big(T_3-S_3d(0)^{-1}c(0)\big) \bigg( - \frac{1}{4\pi} \sum_{j = 1}^2{\langle Q_j, \cdot \rangle Q_j} \bigg)^{-1} T_3 v\gamma_1^* \\
& =2 \big(T_3-S_3d(0)^{-1}c(0)\big) \bigg( - \frac{1}{4\pi} \sum_{j = 1}^2 |Q_j\rangle \;\!\langle Q_j|\bigg)^{-1} T_3 v\gamma_1^*,
\end{align*}
where the standard bra-ket notation has been introduced for the last expression.
\end{Lemma}

\begin{proof}
We begin by recalling some facts about the function $g(\kappa)$ and the operator $m(\kappa)$. Firstly, by \cite[Eqn.~(6.30)]{JN01} the scalar-valued function $g$ satisfies
\begin{align*}
g(\kappa) &= \eta^{-1} \left( - \frac{||v||^2}{2\pi} +\eta h(\kappa) \right)
\end{align*}
where $h$ is bounded near zero. Secondly, by \cite[Eqn.~(6.42)]{JN01}
the operator valued function $m$ satisfies 
\begin{align*}
m(\kappa) &= \eta^{-1} \frac{||v||^2}{8\pi^2} \sum_{j=1}^{d}{\langle Q_j, \cdot \rangle Q_j} + f(\kappa)
\end{align*}
where $f$ is bounded. Then one may write
\begin{align*}
g(\kappa) m(\kappa)^{-1} &= g(\kappa) \left(\eta^{-1}\frac{||v||^2}{8\pi^2} \sum_{j=1}^{2}{\langle Q_j, \cdot \rangle Q_j} + f(\kappa) \right)^{-1} \\
&= \left( \eta^{-1} g(\kappa)^{-1} \frac{||v||^2}{8 \pi^2} \sum_{j=1}^2{\langle Q_j, \cdot \rangle Q_j} + g(\kappa)^{-1} f(\kappa) \right)^{-1},
\end{align*}
and observe that 
\begin{align*}
\lim_{\kappa \to 0}{g(\kappa)^{-1} f(\kappa)} &= \lim_{\kappa \to 0}{\eta \left( - \frac{||v||^2}{2\pi} +\eta h(\kappa)\right)^{-1} f(\kappa)} = 0,
\end{align*}
since $f$ is bounded and $\eta\to 0$ as $\kappa \to 0$. We also have the limit
\begin{align*}
\lim_{\kappa \to 0}{\eta^{-1} g(\kappa)^{-1}} &= \lim_{\kappa \to 0}{\left(  - \frac{||v||^2}{2\pi} +\eta h(\kappa)\right)^{-1}} = - \frac{2\pi}{||v||^2},
\end{align*}
since $h$ is bounded and $\eta \to 0$ as $\kappa \to 0$. Thus we find
\begin{align*}
\lim_{\kappa \to 0}{g(\kappa) m(\kappa)^{-1} } &= \lim_{\kappa \to 0}{\bigg( \eta^{-1} g(\kappa)^{-1} \frac{||v||^2}{8 \pi^2} \sum_{j=1}^2{\langle Q_j, \cdot \rangle Q_j} + g(\kappa)^{-1} f(\kappa) \bigg)^{-1}} \\
&= \bigg( - \frac{1}{4\pi} \sum_{j = 1}^2{\langle Q_j, \cdot \rangle Q_j} \bigg)^{-1}.
\end{align*}

It finally remains to insert the expansion \eqref{eq:asympt} into the expression
for $B(\lambda)$ and we obtain
\begin{align*}
B(0)=  \lim_{\lambda \searrow 0}B(\lambda) 
= & \lim_{\lambda \searrow 0} \varrho^\bot (\lambda) \sqrt{\lambda}\ln(\lambda)S_2 \Big[\big(u+vR_0(\lambda+ i0)v\big)^{-1}v\F_0(\lambda)^*
\Big]  \\
= & \lim_{\lambda \searrow 0} \ln(\lambda) 
\eta S_2\big(T_3-S_3d(\kappa)^{-1}c(\kappa)\big)
g(\kappa)m(\kappa)^{-1}T_3v\gamma_1^* \\
= & 2 \big(T_3-S_3d(0)^{-1}c(0)\big)
\bigg( - \frac{1}{4\pi} \sum_{j = 1}^2{\langle Q_j, \cdot \rangle Q_j} \bigg)^{-1} T_3 v\gamma_1^*
\end{align*}
which leads to the claim.
\end{proof}

Let us compute still more explicitly these expressions. 
Recall firstly that $[\gamma_1 f](\omega):= \frac{-i}{2^{3/2}\pi}\int_{\R^2}\d x\;\!(\omega\cdot x)\;\!f(x)$ for suitable $f\in \ltwo(\R^2)$.
Let us also define $\xi_{\pm 1}\in \hs$ with $\xi_{\pm 1}(\theta):=\frac{1}{\sqrt{2\pi}}\e^{\pm i \theta}$ and $\|\xi_{\pm 1}\|_\hs=1$.
As a consequence, for any $\tau \in \hs$ one has
\begin{align*}
[\gamma_1^*\tau](x) & = \frac{i}{2^{3/2}\pi}\int_{\S}\d \omega\;\! (x\cdot \omega) \;\!\tau(\omega) \\
& = \frac{i}{2^{3/2}\pi}\int_0^{2 \pi}\d \theta\;\! \big(x_1\cos(\theta)+x_2\sin(\theta)\big) \;\!\tau\big((\cos(\theta),\sin(\theta)\big) \\
& = \frac{i}{4\sqrt{\pi}}\int_0^{2\pi}\d \theta \big[x_1 \big(\xi_{1}(\theta)+ \xi_{-1}(\theta)\big) -i 
x_2 \big( \xi_{1}(\theta)- \xi_{-1}(\theta)\big)\big]  \tau\big((\cos(\theta),\sin(\theta)\big)  \\
& = \frac{i}{4\sqrt{\pi}}\big[x_1 \big(\langle \xi_{-1}|+\langle \xi_1|\big) -i 
x_2 \big(\langle \xi_{-1}|-\langle \xi_1|\big)\big]  \tau,
\end{align*}
where $\langle \xi_{\pm 1}|\tau:=\int_0^{2\pi}\d \theta \;\!\overline{\xi_{\pm 1}(\theta)}\;\!
\tau\big((\cos(\theta),\sin(\theta)\big)$.
Since $S_3 v \gamma_1^*=0$, one infers that
\begin{align}
\nonumber T_3 \;\!v\;\!\gamma_1^*  = S_2 \;\!v\;\! \gamma_1^* 
& = \frac{i}{4\sqrt{\pi}} S_2 v \big[X_1 \big(\langle \xi_{-1}|+\langle \xi_1|\big) -i 
X_2 \big(\langle \xi_{-1}|-\langle \xi_1|\big)\big]\\
\nonumber & = \frac{i}{4\sqrt{\pi}} \big[\big|Q_1\big\rangle \big\langle \xi_{-1}+ \xi_1\big| -i \big|Q_2\big\rangle \big\langle \xi_{-1}- \xi_1\big|\big]\\
\label{eq:h2} & = \frac{i}{4\sqrt{\pi}} \big[\big|Q_1-iQ_2\big\rangle \big\langle \xi_{-1}\big|+\big|Q_1+i Q_2\big\rangle\big\langle \xi_1\big|\big].
\end{align}

With these expressions at hand, we can finally compute the expression for 
term $\Gamma_4$ of Proposition  \ref{prop:6c}.
For this, we define an orthogonal projection $P_p$ as follows. If $\dim(T_3)=0$ then $P_p:=0$, 
if  $\dim(T_3)=1$, then 
\begin{equation*}
P_p:=\begin{cases}
\big|\frac{1}{\sqrt{2}}\big(\xi_{-1} - \xi_1\big)\big\rangle  
\big\langle \frac{1}{\sqrt{2}}\big(\xi_{-1}- \xi_1\big)\big|
& \hbox{ if } Q_1=0,\\
\big|\frac{1}{\sqrt{2}}\big(\xi_{-1} + \xi_1\big)\big\rangle  
\big\langle \frac{1}{\sqrt{2}}\big(\xi_{-1}+ \xi_1\big)\big|
& \hbox{ if } Q_2=0, \\
\big|\frac{1}{\sqrt{2(1+|\alpha|^2)}}\big((1+i\overline{\alpha})\xi_{-1}+(1-i\overline{\alpha})\xi_1\big)\big\rangle
\big\langle \frac{1}{\sqrt{2(1+|\alpha|^2)}}\big((1+i\overline{\alpha})\xi_{-1}+(1-i\overline{\alpha})\xi_1\big)\big|
& \hbox{ if } Q_2=\alpha Q_1,  
\end{cases}
\end{equation*}
while if $\dim(T_3)=2$ we set $P_p:=\big(|\xi_{-1}\rangle \langle \xi_{-1}|+ |\xi_{1}\rangle\langle \xi_1| \big)$.

\begin{Lemma}\label{lem:Gamma6}
For any $s\in [+\infty, -\infty]$, one has
\begin{equation*}
\Gamma_4(s)
=\left(\begin{smallmatrix} 
1 & 0 \\ 0 & 1
\end{smallmatrix}\right) -\tfrac{1}{2}\frac{2}{1+i2s}P_p 
\left(\begin{smallmatrix} 1 & 1 \\1 & 1 
\end{smallmatrix}\right) 
= \left(\begin{matrix} 
1 -\frac{1}{1+i2s}P_p & -\frac{1}{1+i2s}P_p \\ 
-\frac{1}{1+i2s}P_p & 1-\frac{1}{1+i2s}P_p
\end{matrix}\right). 
\end{equation*}
\end{Lemma}

\begin{proof}
First of all, observe that 
\begin{align*}
N_2(0)\;\!B(0) & = \gamma_1 v S_2 \;\!B(0) \\
& = 2 \gamma_1 v (T_3+S_3) \big(T_3-S_3d(0)^{-1}c(0)\big) \bigg( - \frac{1}{4\pi} \sum_{j = 1}^2 |Q_j\rangle \;\!\langle Q_j|\bigg)^{-1} T_3 v\gamma_1^* \\
& = 2 \gamma_1 v T_3 \bigg( - \frac{1}{4\pi} \sum_{j = 1}^2 |Q_j\rangle \;\!\langle Q_j|\bigg)^{-1} T_3 v\gamma_1^*
\end{align*}
where the algebraic equality $\gamma_1 v S_3=0$ has been taken into account, see
\cite[Lem.~3.2(c)]{RTZ}.

Then, by using the expression \eqref{eq:h2} for 
$T_3 \;\!v\;\!\gamma_1^*$ and for its adjoint, one infers that the following equality holds:
\begin{align}\label{eq:proj}
\nonumber &2\gamma_1 v T_3\left( - \frac{1}{4\pi} \sum_{j = 1}^2 |Q_j\rangle \;\!\langle Q_j|\right)^{-1} T_3 v\gamma_1^* \\
&=-\frac{1}{2}\big[\big|\xi_{-1}\big\rangle \big\langle Q_1-i Q_2\big| +
\big|\xi_1\big\rangle\big\langle Q_1+i Q_2\big|\big]
T_3\bigg( \sum_{j = 1}^2 |Q_j\rangle \;\!\langle Q_j|\bigg)^{-1}\!\!\!\!T_3  
\big[\big|Q_1-iQ_2\big\rangle \big\langle \xi_{-1}\big|+\big|Q_1+i
Q_2\big\rangle\big\langle \xi_1\big|\big].
\end{align}
By a direct computation, one gets that
 \eqref{eq:proj} is equal to $- P_p$, as defined above. 
Note that for the case $\dim(T_3)=2$, we have used a convenient result due to Parra
about the inversion of a matrix on its range. This statement and its proof are 
gathered in the Appendix.

It only remains to observe that
\begin{equation*}
\Gamma_4(s) 
= 
\left(\begin{smallmatrix} 
1 & 0 \\ 0 & 1
\end{smallmatrix}\right) 
+\tfrac{1}{2}\frac{2}{1+i2s}N_2(0)\;\!B(0)
\left(\begin{smallmatrix} 1 & 1 \\1 & 1 
\end{smallmatrix}\right) 
=  \left(\begin{smallmatrix} 
1 & 0 \\ 0 & 1
\end{smallmatrix}\right) 
-\tfrac{1}{2}\frac{2}{1+i2s} \;\!P_p
\left(\begin{smallmatrix} 1 & 1 \\1 & 1 
\end{smallmatrix}\right),
\end{equation*}
which gives us the statement. 
\end{proof}

\section{Topological Levinson's theorem}\label{sec_Lev}
\setcounter{equation}{0}

In this section, we briefly recall the $C^*$-algebraic framework leading
to a topological version of Levinson's theorem, and show that our current 
investigations fit into this framework.
We refer to the survey paper \cite{Ric16} for additional information on this program
and for the presentation of several examples.

It has already been shown that the unital $C^*$-algebra  
$\big(M_2(\EE)\otimes \K(\hs)\big)^+$ plays the important role of containing
the wave operator $W_-$, once suitable unitary conjugations are applied.
In addition, this algebra contains the ideal $M_2\big(\K(\ltwo(\R_+))\big)\otimes \K(\hs)$
which is nothing but the algebra $\K\big(\ltwo(\R_+;\hs)^2\big)$ of compact operators
on this Hilbert space. Then, as a consequence of Cordes' result, one has the short exact sequence of $C^*$-algebras
$$
0 \longrightarrow \K\big(\ltwo(\R_+;\hs)^2\big) \longrightarrow
\big(M_2(\EE)\otimes \K(\hs)\big)^+ \stackrel{q}{\longrightarrow} 
 \big(M_2\big(C(\hexagon)\big)\otimes \K(\hs)\big)^+ \longrightarrow 0
$$
and the corresponding $6$ terms exact sequence for the $K$-theory of these algebras.
In particular, it is well known that $K_0\Big(\K\big(\ltwo(\R_+;\hs)^2\big)\Big)\cong\Z$
and that $K_1\Big( \big(M_2\big(C(\hexagon)\big)\otimes \K(\hs)\big)^+\Big)\cong\Z$.

Since the operator-valued function $\Gamma= (\Gamma_1,\Gamma_2,\Gamma_3,\Gamma_4,\Gamma_5,\Gamma_6)$ exhibited in Proposition \ref{prop:6c} belongs to 
$\big(M_2\big(C(\hexagon)\big)\otimes \K(\hs)\big)^+$ and is invertible, it defines an element
$[\Gamma]_1$ in the $K_1$-group of this algebra.
In addition, since $W_-\in \big(M_2(\EE)\otimes \K(\hs)\big)^+$ is an isometry and a lift for $\Gamma$, one directly infers from \cite[Prop.~9.2.4.(ii)]{RLL} that
\begin{equation}\label{eq:Lev1}
\ind\big([\Gamma]_1\big) = [1-W^*_-W_-]_0-[1-W_-W_-^*]_0 =-[E_{\rm p}(H)]_0,
\end{equation}
with $E_{\rm p}(H)$ the projection on the subspace spanned by the eigenfunctions of $H$.
Let us emphasize that the equality \eqref{eq:Lev1} corresponds to the topological version of Levinson's theorem: it is a relation (by the index map) between the equivalence class in $K_1$
of quantities related to scattering theory, and the equivalence class in $K_0$ of the projection
on the bound states of $H$. Note that the operator $\Gamma$ contains the scattering operator in its components $\Gamma_2$ and $\Gamma_6$, but also a new contribution related to $p$-resonance in its component $\Gamma_4$.

The standard formulation of Levinson's theorem is an equality between numbers. 
Thus, our last task is to extract a numerical equality from \eqref{eq:Lev1}. In a more general setting we might pair the $K_1$ class of the scattering matrix with the Chern character of a suitable spectral triple, as in \cite{AR}. For this specific case, we proceed in a more elementary way by using the determinant and winding number directly.  
Thus, on $\K\big(\ltwo(\R_+;\hs)^2\big)$, one uses the usual trace (on finite dimensional projections), and on
$$
\big(M_2\big(C(\hexagon)\big)\otimes \K(\hs)\big)^+\cong
\Big(C\big(\hexagon; M_2\big(\K(\hs)\big)\big)\Big)^+
$$
the winding number of the pointwise determinant is the correct notion to be used.

\begin{Remark}\label{rem:conv}
When computing the winding number, and pairing the equality \eqref{eq:Lev1}
with traces, a few conventions about signs have to be taken. As introduced in
\cite[Sec.~2]{Ric16}, we shall turn around the hexagon clockwise, 
and the increase in the winding number is also counted clockwise.
The convention about the path is illustrated in Figure \ref{fig1}, 
with the starting point of $\Gamma_1$ located on the lower left corner.
With this convention, the multiplicative factor $n$, which relates the winding
number computed on $\Gamma$ and the trace applied to $-E_{\rm p}(H)$, is equal to $-1$,
see \cite[Thm.~4.4]{Ric16} for the details.
\end{Remark}

For the computation of the pointwise determinant of the components of $\Gamma$,
let us recall from \cite[Corol.~8.1.7]{Yaf10} that $S(\lambda)-1$ is trace class, 
and that the map $\lambda \mapsto \det\big(S(\lambda)\big)$ is continuous.
Then, based on the following lemma, it will be possible to get simpler expressions
for $\Gamma_1$, $\Gamma_2$, and $\Gamma_6$.

\begin{Lemma}\label{lem:det}
Let $\H$ be a complex Hilbert space and let $c \in \C$ with $|c| = 1$. 
For a unitary operator $U \in \B(\H)$ with $U-1$ trace class, define the operator $B \in \B(\H\oplus\H)$ by
\begin{align*}
B &= \left(\begin{smallmatrix} 
1 & 0 \\ 0 & 1
\end{smallmatrix}\right) +  \tfrac{1}{2} (U-1)
 \left(\begin{smallmatrix} 
1 & c \\ \overline{c} & 1
\end{smallmatrix}\right)
\end{align*}
Then $\sigma(B)\setminus \{1\}=\sigma(U)\setminus\{1\}$, multiplicity counted, 
and $\det(U) = \det(B)$.
\end{Lemma}

\begin{proof}
Let $\lambda$ be an eigenvalue of $B$, with non zero eigenvector $\begin{pmatrix} \xi \\ \eta \end{pmatrix}$, namely $B\begin{pmatrix} \xi \\ \eta \end{pmatrix}=\lambda \begin{pmatrix} \xi \\ \eta \end{pmatrix}$. This equation is equivalent to the two equations
\begin{align*}
\tfrac{1}{2}(U-1)(\xi + c \eta)& =(\lambda-1)\xi, \\
\tfrac{1}{2}(U-1)(\overline{c}\xi + \eta)& =(\lambda-1)\eta. 
\end{align*}
By multiplying the second line by $c$, we infer the relation 
$(\lambda-1)\xi=(\lambda-1)c \eta$. For $\lambda \neq 1$, it follows
that $\eta=\overline{c}\xi$. By inserting this in the first equation, we get
$$
\tfrac{1}{2}(U-1)(\xi + c \eta) = (U-1)\xi = (\lambda-1)\xi,
$$
implying that $U\xi = \lambda \xi$. Note that $\xi\neq 0$, otherwise the eigenvector of $B$ would be the $0$ vector. 

Conversely, if $\xi\neq 0$ satisfies $U\xi =\lambda \xi$, then one easily checks that
the vector $\begin{pmatrix} \xi \\ \overline{c}\xi \end{pmatrix}$ is an eigenvector of $B$
associated with the eigenvalue $\lambda$.
\end{proof}

\begin{Corollary}
One has
\begin{align*}
\det\big(\Gamma_1(s)\big)&=\det\big(S(1)\big), \\
\det\big(\Gamma_2(\ell)\big)&=\det\Big((S\big(\e^{2\ell}\big)\Big), \\
\det\big(\Gamma_6(\ell)\big)&=\det\Big((S\big(\e^{-2\ell}\big)\Big).
\end{align*}
\end{Corollary}

Since one trivially gets $\det\big(\Gamma_3(\xi)\big)=1$ for any $\xi \in [+\infty,0]$, and 
$\det\big(\Gamma_5(\xi)\big)=1$ for any $\xi\in [0,+\infty]$, it only remains to compute $\det\big(\Gamma_4(s)\big)$ for $s\in [+\infty,-\infty]$. However, based on the content
of Lemma \ref{lem:Gamma6} and since $P_p$ is a finite dimensional projection,
this computation is easy. By using again Lemma \ref{lem:det} one infers that
\begin{equation}\label{eq:g4}
\det\big(\Gamma_4(s)\big)= \Big(\frac{i2s-1}{i2s+1}\Big)^{\dim(P_p)}. 
\end{equation}

Before the explicit computation of the winding number of the pointwise determinant, 
it is useful to divide the computation of the Fredholm index of $W_-$ into two contributions.
For that purpose, we define the operator $W_S \in \B(\Hrond)$ by the equality
\begin{equation}\label{eq:dubbya-ess}
W_S-1
:=  \big(\tfrac12\big(1-\tanh(\pi A_+)\big)\otimes 1_\hs\big)\big(S(L)-1\big).
\end{equation}
We then directly obtain its main properties: 

\begin{Lemma}\label{lem_WS}
The operator $W_S$ is a Fredholm operator.
\end{Lemma}

\begin{proof}
It is sufficient to observe that the operator $W_{S^*}$ defines an inverse for $W_S$, 
up to compact operators. Indeed, this can be easily checked by firstly recalling that
$[\vartheta(A_+)\otimes 1_\hs,S(L)]\in \K\big(\ltwo(\R_+;\hs)\big)$, 
with $\vartheta$ defined in \eqref{eq:vartheta}. 
In addition, since  
$S(0) = \lim_{\lambda \to \infty} S(\infty) = 1$ and since $\vartheta-\vartheta^2$ vanishes at $\pm \infty$, operators
of the form $\big(\vartheta(A_+)\otimes 1_\hs-\vartheta^2(A_+)\otimes 1_\hs\big) \big(S(L)-1\big)$
or $\big(\vartheta(A_+)\otimes 1_\hs-\vartheta^2(A_+)\otimes 1_\hs\big) \big(S^*(L)-1\big)$ belong
to $ \K\big(\ltwo(\R_+;\hs)\big)$. 
\end{proof}

For a Freholm operator $W$, let us denote by $\Index(W)$ its Fredholm index. Then the following
statement holds:

\begin{Proposition}\label{propi:Lev}
If $V$ satisfies \eqref{eq_cond_V} with $\rho>11$, then the following equality holds:
\begin{equation}\label{eq:lev2}
\Index(W_S) + \dim(P_p) = -\# \sigma_{\rm p}(H).
\end{equation}
\end{Proposition}

\begin{proof}
Let $W_S$ be as in \eqref{eq:dubbya-ess} and define $W_p := 1-N_2 \Theta B$. 
It follows from \eqref{eq:debase} that
\begin{equation*}
\F_0 W_-\F_0^* 
=    W_S + (W_p-1) + K
\end{equation*}
for a compact operator $K$. 
By construction, the operators $\V\U W_S\U^*\V^*$
and $\V\U (W_p-1) \U^*\V^*$
belong to $\big(M_2(\EE)\otimes \K(\hs)\big)^+$.
For $j \in \{1,2,3,4,5,6\}$ let $\Gamma_{S,j}$ and $\Gamma_{p,j}$ denote the components of the images $q(\V\U W_S\U^*\V^*)$ and $q(\V\U (W_p-1) \U^*\V^* )$ in the quotient algebra. Then a proof similar to Proposition \ref{prop:6c} and Lemma \ref{lem:Gamma6} leads to
\begin{align*}
&\Gamma_{S,1}(s):= 
\left(\begin{smallmatrix} 
1 & 0 \\ 0 & 1
\end{smallmatrix}\right) +  \tfrac{1}{2} \big(S(1)-1\big)
 \left(\begin{smallmatrix} 
1 & \phi(s) \\ \overline{\phi}(s) & 1
\end{smallmatrix}\right),
& s\in [-\infty, +\infty], \\
& \Gamma_{S,2}(\ell):= 
\left(\begin{smallmatrix} 
1 & 0 \\ 0 & 1
\end{smallmatrix}\right) + \tfrac{1}{2}\big(S\big(\e^{2\ell}\big)-1\big)
\left(\begin{smallmatrix} 
1 & -1 \\ -1 & 1
\end{smallmatrix}\right), 
& \ell\in [0,+\infty], \\
&\Gamma_{S,3}(\xi):= \left(\begin{smallmatrix} 1 & 0 \\ 0 & 1 
\end{smallmatrix}\right), & \xi \in [+\infty,0], \\
&\Gamma_{S,4}(s):= 
\left(\begin{smallmatrix} 
1 & 0 \\ 0 & 1
\end{smallmatrix}\right),
& s\in [+\infty, -\infty], \\
&\Gamma_{S,5}(\xi):=\left(\begin{smallmatrix} 1 & 0 \\ 0 & 1 
\end{smallmatrix}\right),  & \xi \in [0,+\infty], \\
&\Gamma_{S,6}(\ell):= 
\left(\begin{smallmatrix} 
1 & 0 \\ 0 & 1
\end{smallmatrix}\right) + \tfrac{1}{2}\big(S\big(\e^{-2\ell}\big)-1\big)
\left(\begin{smallmatrix} 
1 & 1 \\ 1 & 1
\end{smallmatrix}\right), 
& \ell \in [+\infty,0],
\end{align*} 
and to 
\begin{align*}
&\Gamma_{p,j}:= 
\left(\begin{smallmatrix} 
0 & 0 \\ 0 & 0
\end{smallmatrix}\right) ,
& j \in \{1,2,3,5,6\}, \\
&\Gamma_{p,4}(s):=  -\tfrac{1}{2}\frac{2}{1+i2s}P_p 
\left(\begin{smallmatrix} 1 & 1 \\1 & 1 
\end{smallmatrix}\right) ,
& s\in [+\infty, -\infty].
\end{align*}

Explicit computation shows that $\Gamma_{S,j}^* \Gamma_{p,j} = \Gamma_{p,j} = 0$ for $j \neq 4$ and $\Gamma_{S,4}^* \Gamma_{p,4} = \Gamma_{p,4}$. Thus we find  $q(\V\U(W_S^* (W_p-1)\V^*\U^* ) = q(\V\U (W_p-1)\U^*\V^*)$. Since their image under the quotient map agrees, we have $1+W_S^*(W_p-1) = W_p+K$ for some compact operator $K$. 
Now we note the equalities
\begin{align*}
\F_0 W_-\F_0^* &= W_S + (W_p-1) + K = W_S\big(1+ W_S^*(W_p-1)\big) + K',
\end{align*}
for some compact operators $K$ and $K'$, 
which lead to 
\begin{align*}
\Index(W_-) &= \Index(W_S) + \Index\big(1+W_S^*(W_p-1)\big) = \Index(W_S)+\Index(W_p).
\end{align*}

Clearly, one has $\Index(W_-)=-\# \sigma_{\rm p}(H)$.
On the other hand, the value $\Index(W_p)$ can be computed with the winding number 
of the poinwise determinant of $1+\Gamma_{p,4}$, as mentioned in Remark \ref{rem:conv}.
More precisely, one has  
$$
\Index(W_p)= -\Wind\big(\det(1+\Gamma_{p,4})\big)=  -\Wind\big(\det(\Gamma_{4})\big)
$$
with $\det(\Gamma_{4})$ provided in \eqref{eq:g4}.
However, since $s\mapsto \det\big(\Gamma_{4}(s)\big)$ has to be computed from $+\infty$
to $-\infty$, and that on this path the increase of the argument is anti-clockwise, one gets
$\Index(W_p)=  \dim(P_p)$, leading directly to the statement. 
\end{proof}

Our next aim is to compute $\Index(W_-)$ in terms of $\Gamma_S$, as introduced in 
the proof of Proposition \ref{propi:Lev}.
Due to the high energy behaviour of the scattering matrix, some regularization is necessary to obtain an analytic formula. Indeed, even though the map $\lambda\mapsto S(\lambda)$
converges to $1$ in the norm on $\B(\hs)$, the map $\lambda \mapsto \det\big(S(\lambda)\big)$ 
does not converge to $1$. A more precise statement is provided in Lemma \ref{lem:det-loop}.

For $\lambda \in \R_+$ we define the self-adjoint operator $A(\lambda)$ in $\B(\hs)$ by 
\begin{align*}
A(\lambda) &= 4 \tan^{-1}(\lambda) \F_0(\lambda) V \F_0(\lambda)^*.
\end{align*}
The main properties of this operator are gathered in the following statement.

\begin{Lemma}
For each $\lambda \in \R_+$ the operator $A(\lambda)$ is self-adjoint and trace class. 
\end{Lemma}

\begin{proof}
Since $V$ is real-valued, the self-adjointness property is clear. 
Based on the definition of $\F_0$ given in \eqref{eq_diag}, we can write explicitly 
the integral kernel of $A(\lambda)$ as
\begin{align*}
A(\lambda,\omega,\omega') &= 2 \tan^{-1}(\lambda) (2\pi)^{-2} \int_{\R^2} e^{-i\sqrt{\lambda} (\omega-\omega') \cdot x} V(x) \, \d x.
\end{align*}
Integrating along the diagonal shows that the trace of $A(\lambda)$ is
\begin{align}\label{eq:trace}
\Tr\big(A(\lambda)\big) &= \frac{1}{\pi} \tan^{-1}(\lambda) \int_{\R^2} V(x)\, \d x.
\end{align}
The computation is justified by writing
\begin{equation}\label{eq:prodH}
4 \tan^{-1}(\lambda) \F_0(\lambda) v u v \F_0(\lambda)^*
\end{equation}
with $u,v$ introduced in \eqref{eq:uv}, and by observing 
that \eqref{eq:prodH} constains two factors which are Hilbert-Schmidt. 
\end{proof}

By the properties of the map $\lambda \mapsto \F_0(\lambda)$ exhibited in \cite[Lem.~4.8]{RTZ}
one infers that the operator-valued map $\lambda\mapsto A(\lambda)\in \B(\hs)$ is continuous and has norm limits
$\lim_{\lambda\searrow 0} A(\lambda)=0$ and $\lim_{\lambda \to \infty}A(\lambda)=0$. As a consequence of \eqref{eq:trace}, one also observes that the map 
$\lambda \mapsto \Tr\big(A(\lambda)\big)$ is continuous on $\R_+$ and satisfies
$\lim_{\lambda\searrow 0} \Tr\big(A(\lambda)\big) = 0$ and 
$\lim_{\lambda \to \infty}\Tr\big(A(\lambda)\big) = \frac12 \int_{\R^2} V(x) \d x$.
 
Based on these observations, let us now define the unitary operator in $\B(\hs)$
$$
\beta(\lambda) := \exp\big(i A(\lambda)\big) 
$$ 
which clearly satisfies
$\det\big(\beta(\lambda)\big) = \e^{i \Tr(A(\lambda))}$
for all $\lambda \in \R_+$.  We also define the operator $W_\beta \in \B(\Hrond)$ by the equality
\begin{equation*}
W_\beta-1 
=   \big(\tfrac12\big(1-\tanh(\pi A_+)\big)\otimes 1_\hs\big)\big(\beta(L)-1\big).
\end{equation*}
Our main interest for this operator is related to the properties shown in the next statement.

\begin{Lemma}\label{lem_Wb}
The operator $W_\beta$ is a Fredholm operator satisfying $\Index(W_\beta) = 0$.
\end{Lemma}

\begin{proof}
Observe firstly that we have the norm limits $\lim_{\lambda \searrow 0}\beta(\lambda) = 1$ and that 
$\lim_{\lambda \to \infty} \beta(\lambda) = 1$.
Thus, by the same argument provided in the proof of Lemma \ref{lem_WS} one gets that
the operator $W_{\beta^*}$ defines an inverse for $W_\beta$, up to compact operators. 
It directly follows that $W_\beta$ is a Fredholm operator.

It remains to show that $\Index(W_\beta) = 0$. To see this we consider, for fixed $\lambda \in \R_+$, the map $[0,1] \ni t \mapsto A_t(\lambda)$ with $A_t(\lambda)$ defined by
\begin{align*}
A_t(\lambda) &= 4 \tan^{-1}((1-t) \lambda) \F_0(\lambda) V \F_0(\lambda)^*.
\end{align*}
The map $A_t(\lambda)$ defines a norm continuous path in $\B(\hs)$ from $A(\lambda)$ to $0$. Defining the path $A_t=A_t(L)\in\B(\Hrond)$ 
we then obtain a norm continuous path in $\B(\Hrond)$ from $A$ to $0$. As a consequence, $\beta_t = \exp( i A_t)$ defines a norm continuous path of unitary operators in $\B(\Hrond)$ from $\beta$ to $1$. Hence the path $W_{\beta_t}$ defines a norm continuous path in $\B(\Hrond)$ from $W_\beta$ to the identity, along which the Fredholm index is constant, and so equal to $0$.
\end{proof}

\begin{Lemma}\label{lem:equal-index}
The Fredholm operators $W_S$ and $W_\beta$ satisfy
$\ W_S W_{\beta} - W_{S\beta}  \in \K(\Hrond)\ $
and
$$
\Index(W_S) = \Index(W_{S\beta}).
$$ 
\end{Lemma}

\begin{proof}
The equality $W_S W_\beta = W_{S\beta}$ up to compact operators follows from one more commutator computation as provided in the proof of Lemmas \ref{lem_WS} and \ref{lem_Wb}. 
The index claim follows from the fact that $\Index(W_\beta) = 0$ and the composition rule for Fredholm index.
\end{proof}

The next statement shows that the operator $\beta$ provides the correct regularization for the operator $S$, and consequently $W_\beta$ will provide the correct regularization
to the operator $W_S$.
The proof is using some properties of the spectral shift function developed in \cite[Chap.~9]{Yaf10}.

\begin{Lemma}\label{lem:det-loop}
The map $\lambda \mapsto \det\big(S(\lambda)\big)\det\big(\beta(\lambda)\big)$ satisfies
\begin{equation*}
\lim_{\lambda \searrow 0} \det\big(S(\lambda)\big)\det\big(\beta(\lambda)\big) = 1
\end{equation*}
and 
\begin{equation*}
\lim_{\lambda \to \infty} \det\big(S(\lambda)\big)\det\big(\beta(\lambda)\big) = 1.
\end{equation*}
\end{Lemma}

\begin{proof}
Let us firstly recall the Birman-Kre\u {\i}n formula linking the scattering operator and the spectral shift function, namely $\det\big(S(\lambda)\big)=\e^{-2\pi i \xi(\lambda)}$.
By \cite[Thm.~9.1.14]{Yaf10} there exists a continuous function $\xi_2$ (the regularised spectral shift function) such that
\begin{align*}
\xi(\lambda) &= \xi_2(\lambda) + \frac{1}{4\pi} \int_{\R^2} V(x) \, \d x,
\end{align*}
with $\lim_{\lambda \to \infty}\xi_2(\lambda) = 0$.
In addition, since $S(0)=1$, it follows that $\lim_{\lambda \searrow 0 }\xi(\lambda) \in \Z$. We can thus consider the function $\lambda \mapsto f(\lambda)$ with 
\begin{align*}
f(\lambda):=& -2\pi i\xi(\lambda) +i \Tr \big(A(\lambda)\big) \\
=& -2 \pi i \Big[\xi_2(\lambda) + \frac{1}{4\pi} \left(1- \frac{2}{\pi} \tan^{-1}(\lambda) \right) \int_{\R^2} V(x)\, \d x \Big],
\end{align*}
which satisfies $\lim_{\lambda \to \infty}f(\lambda)=0$
and $\lim_{\lambda \searrow 0}f(\lambda) \in (-2\pi i) \Z$.
It finally remains to observe that
\begin{equation*}
 \det\big(S(\lambda)\big)\det\big(\beta(\lambda)\big)  = \e^{-2\pi i \xi(\lambda)}
 \e^{i\Tr (A(\lambda))}= \e^{f(\lambda)}
\end{equation*}
and so the map $\lambda \mapsto \det\big(S(\lambda)\big)\det\big(\beta(\lambda)\big)$
satisfies the properties stated. 
\end{proof}

We finally recall from \cite[Eq.~(9.1.22)]{Yaf10} that the spectral shift function $\xi$ satisfies for $\lambda > 0$ the equality
\begin{equation}\label{eq:SS'}
\Tr\big(S(\lambda)^* S'(\lambda)\big) = -2\pi i \xi'(\lambda),
\end{equation}
with the differentiability of $\xi$ being guaranteed by \cite[Thm.~9.1.18]{Yaf10}.
We can thus state the main result of this section.

\begin{Proposition}\label{prop:comput}
The following equality holds: 
\begin{equation*}
\Index(W_S) = \frac{1}{2\pi i} \int_0^\infty \Tr\big(S(\lambda)^*S'(\lambda)\big) \, \d \lambda + \frac{1}{4\pi} \int_{\R^2} V(x)\, \d x.
\end{equation*}
\end{Proposition}

\begin{proof}
By Lemma \ref{lem:equal-index} one has $\Index(W_S) = \Index(W_{S\beta})$. 
Note that by \cite[Thm.~3.5(a)]{simon} we have $\det(S\beta) = \det(S) \det(\beta)$. By Lemma \ref{lem:det-loop}, $\det(S\beta)$ defines a loop, and using Gohberg-Kre\u{\i}n theory (cf. \cite{GK} and \cite[Thm.~4.9]{Lesch}) we can compute the index of $W_{S\beta}$ as
\begin{align*}
\Index(W_{S\beta}) &= \Wind\big(\det(S\beta)\big) \\
&= \frac{1}{2\pi i} \int_0^\infty \frac{\frac{\d}{\d \lambda}\big[ \det\big(S(\lambda)\big) \det\big(\beta(\lambda)\big)\big]}{\det\big(S(\lambda)\big) \det\big(\beta(\lambda)\big)} \, \d \lambda \\
&= \frac{1}{2\pi i} \int_0^\infty \frac{\d}{\d \lambda} \Big( -2\pi i\xi(\lambda) +i \Tr\big(A(\lambda)\big) \Big) \, \d \lambda \\
&= \frac{1}{2\pi i} \int_0^\infty \left( -2\pi i \xi'(\lambda) \right)\, \d \lambda + \frac{1}{2\pi} \left[\Tr\big(A(\infty)\big) - \Tr\big(A(0)\big) \right] \\
&= \frac{1}{2\pi i} \int_0^\infty \Tr\big(S(\lambda)^* S'(\lambda)\big) \, \d \lambda + \frac{1}{4\pi} \int_{\R^2} V(x) \, \d x,
\end{align*}
as claimed.
\end{proof}

By collecting the content of Proposition \ref{propi:Lev} and of Proposition \ref{prop:comput},
we can now confirm the statement of \cite[Thm.~6.3]{BGD88}, namely:

\begin{Theorem}\label{thm:analytic-formula}
If $V$ satisfies \eqref{eq_cond_V} with $\rho>11$, then the following equality holds:
\begin{equation*}
\frac{1}{2\pi i} \int_0^\infty \Tr\big(S(\lambda)^*S'(\lambda)\big) \, \d \lambda + \frac{1}{4\pi} \int_{\R^2} V(x)\, \d x + \dim(P_p) = -\# \sigma_{\rm p}(H).
\end{equation*}
\end{Theorem}

We can then complement the content of \cite[Thm.~9.1.14]{Yaf10} about the spectral shift function.

\begin{Corollary}
If $V$ satisfies \eqref{eq_cond_V} with $\rho>11$, then the spectral shift function for the pair $(H,H_0)$ satisfies 
\begin{align*}
\lim_{\varepsilon \searrow 0}{\xi(\varepsilon)} &= -\# \sigma_{\rm p}(H) - \dim(P_p).
\end{align*}
\end{Corollary}

\begin{proof}
By \cite[Thm.~9.1.14]{Yaf10} we have
\begin{align*}
\xi(\infty) &= \frac{1}{4\pi} \int_{\R^2}{V(x)\, \d x}.
\end{align*}
By taking the equality \eqref{eq:SS'} into account, we observe that 
\begin{align*}
\lim_{\varepsilon \searrow 0}{\xi(\varepsilon)} &= - \int_0^\infty{\xi'(\lambda)\, \d \lambda} + \xi(\infty) \\
&= \frac{1}{2\pi i} \int_0^\infty{\Tr\big(S(\lambda)^*S'(\lambda)\big) \, \d \lambda} + \frac{1}{4\pi} \int_{\R^2}{V(x)\, \d x}.
\end{align*}
The result now follows from Theorem \ref{thm:analytic-formula}.
\end{proof}

\section*{Appendix}
 
In the following statement, the standard bra-ket notation is freely used.

\begin{Lemma}\label{inversodaga}
Let $\H$ be a complex Hilbert space, and let $\varphi,\psi\in \H$ be linearly independent. 
Consider $c\in\C$ with $|c|=1$, define $T:\H\to \H$ by
\begin{equation*}
T:=|\varphi\rangle \langle\varphi|+c|\psi\rangle \langle\psi| , 
\end{equation*}
and set 
\begin{equation*}
k:=||\varphi||^2\,||\psi||^2-|\langle\varphi,\psi\rangle|^2 >0.
\end{equation*}
Then, the operator $T^\dagger$ defined by 
\begin{align*}
T^\dagger:=\frac{1}{ck^2}\Big[\big(c||\psi||^4+|\langle\varphi,\psi\rangle|^2\big)&|\varphi\rangle \langle\varphi| -\big(c||\psi||^2\langle\varphi,\psi\rangle+||\varphi||^2\langle\varphi,\psi\rangle\big)|\varphi\rangle \langle\psi|\\
-&\big(c||\psi||^2\langle\psi,\varphi\rangle+||\varphi||^2\langle\psi,\varphi\rangle\big)|\psi\rangle \langle\varphi|+\big(||\varphi||^4+c|\langle\varphi,\psi\rangle|^2\big)|\psi\rangle \langle\psi|  \Big]
\end{align*}
satisfies $TT^\dagger=T^\dagger T=P_{\Ran (T)}$, the projection on the range of $T$.
Furthermore, the following equalities 
hold:\ 
$\langle \varphi,  T^\dagger \varphi \rangle = 1$, \ $\langle \psi, T^\dagger  \psi \rangle  = \overline{c}$,  \ $\langle \varphi, T^\dagger \psi \rangle =0$, and 
$\langle \psi, T^\dagger  \varphi \rangle  =0$.
\end{Lemma}

\begin{proof}
We first observe that 
\begin{equation*}
T\varphi=||\varphi||^2\varphi+c\langle\psi,\varphi\rangle \psi\qquad T\psi=\langle\varphi,\psi\rangle\varphi+c||\psi||^2\psi\ .
\end{equation*}
We can also compute
\begin{align*}
T^\dagger\varphi=&\frac1{ck^2}\Bigl[\Bigl(\big(c||\psi||^4+|\langle\varphi,\psi\rangle|^2\big)||\varphi||^2-\big(c||\psi||^2\langle\varphi,\psi\rangle+||\varphi||^2\langle\varphi,\psi\rangle\big)\langle\psi,\varphi\rangle\Bigr) \varphi \\
\phantom{=}&+\Bigl(-\big(c||\psi||^2\langle\psi,\varphi\rangle+||\varphi||^2\langle\psi,\varphi\rangle\big)||\varphi||^2+\big(||\varphi||^4+c|\langle\varphi,\psi\rangle|^2\big)\langle\psi,\varphi\rangle \Bigr)\psi\Bigr] \\
=&\frac1{ck^2}\Bigl[c||\psi||^2\big(||\psi||^2||\varphi||^2-|\langle\psi,\varphi\rangle|^2\big) \varphi +c\langle\psi,\varphi\rangle\big(|\langle\varphi,\psi\rangle|^2-||\psi||^2||\varphi||^2 \big)\psi\Bigr] \\
=&\frac1k\bigl[||\psi||^2\varphi-\langle\psi,\varphi\rangle\psi\bigr],
\end{align*}
and similarly one gets
\begin{align*}
T^\dagger \psi=& \frac1{ck^2}\Bigl[\Bigl(\big(c||\psi||^4+|\langle\varphi,\psi\rangle|^2\big)\langle\varphi,\psi\rangle-\big(c||\psi||^2\langle\varphi,\psi\rangle+||\varphi||^2\langle\varphi,\psi\rangle\big)||\psi||^2\Bigr) \varphi \\
\phantom{=}&+\Bigl(-\big(c||\psi||^2\langle\psi,\varphi\rangle+||\varphi||^2\langle\psi,\varphi\rangle\big)\langle\varphi,\psi\rangle+\big(||\varphi||^4+c|\langle\varphi,\psi\rangle|^2\big)||\psi||^2 \Bigr)\psi\Bigr] \\
=&\frac1{ck^2} \Bigl[\langle\varphi,\psi\rangle\big(|\langle\varphi,\psi\rangle|^2-||\psi||^2||\varphi||^2\big)\varphi+||\varphi||^2\big(||\psi||^2||\varphi||^2-|\langle\psi,\varphi\rangle|^2 \big)\psi\Bigr]\\
=&\frac1{ck}\big[-\langle\varphi,\psi\rangle\varphi+||\varphi||^2\psi\big].
\end{align*}

From these we obtain that
\begin{align*}
T^\dagger T\varphi=&||\varphi||^2T^\dagger \varphi+c\langle\psi,\varphi\rangle T^\dagger \psi\\
=&\frac{||\varphi||^2}{k}\bigl[||\psi||^2\varphi-\langle\psi,\varphi\rangle\psi\bigr]+\frac{c\langle\psi,\varphi\rangle }{ck}\big[-\langle\varphi,\psi\rangle\varphi+||\varphi||^2\psi\big] \\
=& \varphi,
\end{align*}
and
\begin{align*}
T T^\dagger \varphi=&\frac1k\bigl[||\psi||^2T \varphi-\langle\psi,\varphi\rangle T\psi\bigr]\\
=&\frac1k\Bigl[||\psi||^2\big(||\varphi||^2\varphi+c\langle\psi,\varphi\rangle \psi\big)-\langle\psi,\varphi\rangle \big(\langle\varphi,\psi\rangle\varphi+c||\psi||^2\psi\ \big)\Bigr] \\
=& \varphi.
\end{align*}
By a similar computation, one also gets $T^\dagger T\psi=\psi$ and $T T^\dagger \psi = \psi$.
The remaining equalities can also be obtained straightforwardly.
\end{proof}



\end{document}